\documentclass[reqno]{amsart}

\usepackage[letterpaper,top=2cm,bottom=2cm,left=3cm,right=3cm,marginparwidth=2.0cm]{geometry}

\usepackage{amsmath,amsthm,amssymb,mathtools, mathrsfs,booktabs} 
\usepackage[english]{babel}
\usepackage[utf8]{inputenc}
\usepackage{microtype} 
\usepackage{csquotes}
\usepackage{graphicx}
\usepackage{ytableau}
\usepackage{tikz}
\usepackage{verbatim}
\usetikzlibrary{decorations.pathreplacing}
\usepackage{todonotes}
\presetkeys{todonotes}{inline}{}
\usepackage{arydshln}

\usepackage[colorlinks=true, allcolors=blue]{hyperref}

\theoremstyle{definition}
\newtheorem{example}{Example}
\newtheorem{definition}{Definition}
\newtheorem{remark}{Remark}

\theoremstyle{plain}
\newtheorem{theorem}{Theorem}
\newtheorem{lemma}{Lemma}
\newtheorem{proposition}{Proposition}
\newtheorem{corollary}{Corollary}
\newtheorem{conjecture}{Conjecture}

\newcommand{\oeis}[1]{\href{https://oeis.org/#1}{#1}}
\definecolor{lightblue}{RGB}{173, 216, 230}
\definecolor{skyblue}{RGB}{135, 206, 255}
\definecolor{pastelblue}{RGB}{176, 224, 230}
\definecolor{iceblue}{RGB}{200, 230, 250}

\definecolor{mintgreen}{RGB}{144, 238, 144}     
\definecolor{pastelgreen}{RGB}{152, 251, 152}   
\definecolor{softgreen}{RGB}{193, 255, 193} 
\definecolor{lemonyellow}{RGB}{255, 250, 205}
\definecolor{peachpink}{RGB}{255, 218, 185}
\definecolor{lavender}{RGB}{230, 230, 250}
\definecolor{aquablue}{RGB}{175, 238, 238}

\newcommand{\defin}[1]{%
\relax\ifmmode%
\textcolor{blue}{#1}%
\else\textcolor{blue}{\emph{#1}}%
\fi%
}

\newcommand{\setN}{\mathbb{N}}

\newcommand{\setZ}{\mathbb{Z}}

\newcommand{\uvec}{\mathbf{u}}

\newcommand{\wvec}{\mathbf{w}}
\newcommand{\xvec}{\mathbf{x}}
\newcommand{\yvec}{\mathbf{y}}

\newcommand{\symS}{\mathfrak{S}} 

\newcommand{\versch}{\varphi} 
\DeclareMathOperator{\SH}{\mathrm{SH}}

\DeclareMathOperator{\std}{std}
\DeclareMathOperator{\DES}{DES}

\DeclareMathOperator{\length}{\ell}

\DeclareMathOperator{\lettercount}{count}
\DeclareMathOperator{\rowMap}{\mathtt{rowDiv}}
\newcommand{\rowAdj}[1]{\mathop{\mathtt{rowDiv}_{#1}^{\perp}}}
\DeclareMathOperator{\colMap}{\mathtt{colDiv}}
\newcommand{\colAdj}[1]{\mathop{\mathtt{colDiv}_{#1}^{\perp}}}

\newcommand{\mittagLefflerE}{\mathcal E} 
\newcommand{\stanleyMap}{\Theta} 

\newcommand{\schur}{\mathrm{s}}
\newcommand{\monomial}{\mathrm{m}}
\newcommand{\powersum}{\mathrm{p}}
\newcommand{\completeH}{\mathrm{h}}
\newcommand{\elementaryE}{\mathrm{e}}
\newcommand{\gessel}{F} 
\newcommand{\qmonomial}{M} 

\newcommand{\dominatesBy}{\unlhd}

\newcommand{\YSSYT}{\mathrm{YSSYT}}
\newcommand{\SSYT}{\mathrm{SSYT}}
\newcommand{\SYT}{\mathrm{SYT}}
\newcommand{\QSym}{\mathrm{QSym}}
\newcommand{\BCM}{\mathrm{BCM}}
\newcommand{\sgn}{\mathrm{sgn}}
\newcommand{\SRHT}{\mathrm{SRHT}}

\makeatletter

\ExplSyntaxOn
\seq_new:N \l__rht_rows_seq
\int_new:N \l__rht_rowcount_int
\bool_new:N \l__rht_french_bool
\bool_new:N \l__rht_dot_bool
\tl_new:N \l__rht_scale_tl
\tl_new:N \l__rht_currow_tl
\tl_new:N \l__rht_curdigit_tl
\tl_new:N \l__rht_neighbor_tl
\int_new:N \l__rht_curcols_int
\NewDocumentCommand{\rimHookTableau}{ O{} m }
{
  \tl_set:Nn \l__rht_scale_tl { 1.0 }
  \bool_set_false:N \l__rht_french_bool
  \bool_set_false:N \l__rht_dot_bool
  \keys_set:nn { rht } { #1 }
  \__rht_parse_and_draw:n { #2 }
}

\keys_define:nn { rht }
{
  scale .tl_set:N = \l__rht_scale_tl,
  French .bool_set:N = \l__rht_french_bool,
  Dot .bool_set:N = \l__rht_dot_bool,
}

\cs_new_protected:Nn \__rht_parse_and_draw:n
{
  \seq_set_split:Nnn \l__rht_rows_seq { , } { #1 }
  \int_set:Nn \l__rht_rowcount_int { \seq_count:N \l__rht_rows_seq }
  \__rht_store_cells:
  \__rht_build_commands:
  \__rht_draw:
}

\cs_new_protected:Nn \__rht_store_cells:
{
  \int_step_inline:nn { \l__rht_rowcount_int }
  {
    \tl_set:Nx \l__rht_currow_tl { \seq_item:Nn \l__rht_rows_seq { ##1 } }
    \int_set:Nn \l__rht_curcols_int { \tl_count:N \l__rht_currow_tl }
    \tl_gset:cx { g__rht_rowlen_##1 } { \int_use:N \l__rht_curcols_int }
    \int_step_inline:nn { \l__rht_curcols_int }
    {
      \tl_gset:cx { g__rht_cell_##1_####1 } 
        { \tl_item:Nn \l__rht_currow_tl { ####1 } }
    }
  }
}

\cs_new:Nn \__rht_get_cell:nn
{
  \tl_if_exist:cTF { g__rht_cell_#1_#2 }
    { \tl_use:c { g__rht_cell_#1_#2 } }
    { }
}

\cs_new:Nn \__rht_get_rowlen:n
{
  \tl_if_exist:cTF { g__rht_rowlen_#1 }
    { \tl_use:c { g__rht_rowlen_#1 } }
    { 0 }
}

\cs_new:Nn \__rht_ycoord:n
{
  \bool_if:NTF \l__rht_french_bool
    { \int_eval:n { \l__rht_rowcount_int - #1 + 1 } }
    { #1 }
}

\cs_new_protected:Nn \__rht_build_commands:
{
  \gdef\rhtcmds{}
  
  \int_step_inline:nn { \l__rht_rowcount_int }
  {
    \int_step_inline:nn { \__rht_get_rowlen:n { ##1 } }
    { \__rht_add_cell_cmd:nn { ##1 } { ####1 } }
  }
  
  \int_step_inline:nn { \l__rht_rowcount_int }
  {
    \int_step_inline:nn { \__rht_get_rowlen:n { ##1 } }
    { \__rht_add_connection_cmd:nn { ##1 } { ####1 } }
  }
  
  \bool_if:NT \l__rht_dot_bool
  {
    \int_step_inline:nn { \l__rht_rowcount_int }
    {
      \int_step_inline:nn { \__rht_get_rowlen:n { ##1 } }
      { \__rht_add_dot_cmd:nn { ##1 } { ####1 } }
    }
  }
}

\cs_new_protected:Nn \__rht_add_cell_cmd:nn
{
  \tl_set:Nx \l__rht_curdigit_tl { \__rht_get_cell:nn{#1}{#2} }
  \str_if_eq:VnF \l__rht_curdigit_tl { 0 }
  {
    \edef\rht@tempa{\int_eval:n{#2-1}}
    \edef\rht@tempb{\int_eval:n{-\__rht_ycoord:n{#1}}}
    \exp_args:NNx \g@addto@macro \rhtcmds { \exp_not:N \rhtcell{\rht@tempa}{\rht@tempb} }
  }
}

\cs_new_protected:Nn \__rht_add_connection_cmd:nn
{
  \tl_set:Nx \l__rht_curdigit_tl { \__rht_get_cell:nn{#1}{#2} }
  \tl_if_empty:NF \l__rht_curdigit_tl
  {
    \str_if_eq:VnF \l__rht_curdigit_tl { 0 }  
    {
      \int_compare:nNnT {#2} < { \__rht_get_rowlen:n{#1} }
      {
        \tl_set:Nx \l__rht_neighbor_tl { \__rht_get_cell:nn{#1}{\int_eval:n{#2+1}} }
        \tl_if_eq:NNT \l__rht_curdigit_tl \l__rht_neighbor_tl
        { \__rht_add_hline_cmd:nn{#1}{#2} }
      }
      \bool_if:NTF \l__rht_french_bool
      {
        \int_compare:nNnT {#1} > {1}
        {
          \int_compare:nNnT {#2} < { \int_eval:n{ \__rht_get_rowlen:n{\int_eval:n{#1-1}} + 1 } }
          {
            \tl_set:Nx \l__rht_neighbor_tl { \__rht_get_cell:nn{\int_eval:n{#1-1}}{#2} }
            \tl_if_eq:NNT \l__rht_curdigit_tl \l__rht_neighbor_tl
            { \__rht_add_vline_cmd:nn{#1}{#2} }
          }
        }
      }
      {
        \int_compare:nNnT {#1} < {\l__rht_rowcount_int}
        {
          \int_compare:nNnT {#2} < { \int_eval:n{ \__rht_get_rowlen:n{\int_eval:n{#1+1}} + 1 } }
          {
            \tl_set:Nx \l__rht_neighbor_tl { \__rht_get_cell:nn{\int_eval:n{#1+1}}{#2} }
            \tl_if_eq:NNT \l__rht_curdigit_tl \l__rht_neighbor_tl
            { \__rht_add_vline_cmd:nn{#1}{#2} }
          }
        }
      }
    }
  }
}

\cs_new_protected:Nn \__rht_add_hline_cmd:nn
{
  \edef\rht@tempa{\int_eval:n{#2}}
  \edef\rht@tempb{\int_eval:n{-\__rht_ycoord:n{#1}}}
  \exp_args:NNx \g@addto@macro \rhtcmds { \exp_not:N \rhthline{\rht@tempa}{\rht@tempb} }
}

\cs_new_protected:Nn \__rht_add_vline_cmd:nn
{
  \edef\rht@tempa{\int_eval:n{#2}}
  \edef\rht@tempb{\int_eval:n{-\__rht_ycoord:n{#1}}}
  \bool_if:NTF \l__rht_french_bool
  {
    \edef\rht@tempc{\int_eval:n{-\__rht_ycoord:n{\int_eval:n{#1-1}}}}
  }
  {
    \edef\rht@tempc{\int_eval:n{-\__rht_ycoord:n{\int_eval:n{#1+1}}}}
  }
  \exp_args:NNx \g@addto@macro \rhtcmds { \exp_not:N \rhtvline{\rht@tempa}{\rht@tempb}{\rht@tempc} }
}

\cs_new_protected:Nn \__rht_add_dot_cmd:nn
{
  \tl_set:Nx \l__rht_curdigit_tl { \__rht_get_cell:nn{#1}{#2} }
  \str_if_eq:VnF \l__rht_curdigit_tl { 0 }
  {
    \edef\rht@tempa{\int_eval:n{#2}}
    \edef\rht@tempb{\int_eval:n{-\__rht_ycoord:n{#1}}}
    \exp_args:NNx \g@addto@macro \rhtcmds { \exp_not:N \rhtdot{\rht@tempa}{\rht@tempb} }
  }
}

\cs_new_protected:Nn \__rht_draw:
{
  \edef\rhtscale{\l__rht_scale_tl}
  \rhtdrawtableau
}

\ExplSyntaxOff
\def\rhtcell#1#2{\draw (#1,#2) rectangle +(1,1);}
\def\rhthline#1#2{\draw[line width=1.5pt] (#1-0.5,#2+0.5) -- (#1+0.5,#2+0.5);}
\def\rhtvline#1#2#3{\draw[line width=1.5pt] (#1-0.5,#2+0.5) -- (#1-0.5,#3+0.5);}
\def\rhtdot#1#2{\fill (#1-0.5,#2+0.5) circle (2pt);}

\def\rhtdrawtableau{%
  \begin{tikzpicture}[scale=\rhtscale,baseline=(current bounding box.center)]
    \rhtcmds
  \end{tikzpicture}%
}

\makeatother

\title{Partition division maps, symmetric functions and positivity}

\author{Per Alexandersson}
\address{Department of Mathematics, Stockholm University, SE-106 91 Stockholm, Sweden}
\email{per.w.alexandersson@gmail.com}

\author{Lilan Dai}
\address{Center for Combinatorics, LPMC, 
	Nankai University, Tianjin 300071, P. R. China}
\email{daililan@mail.nankai.edu.cn}

\newcommand{\QQ}{\mathbb Q}

\begin{document}

\begin{abstract}
For a positive integer $k$, the partition division map $\rowMap_k$ sends $\monomial_{k\lambda}$ to $\monomial_\lambda$ and all other monomial symmetric functions to zero. We describe the Schur expansion of $\rowMap_k(\schur_{\lambda/\mu})$ in terms of $k$-Yamanouchi tableaux, generalizing the Littlewood--Richardson rule. For products of elementary symmetric functions, we obtain Schur and power-sum expansions. Using cyclic words, we prove that $\omega(\rowMap_k(\elementaryE_\mu))$ is power-sum positive. We also give a block-tableau interpretation of sums of elementary coefficients over partitions of a fixed length, and establish connections with alternating permutations and the symmetric functions studied by Amdeberhan, Shareshian, and Stanley.
\medskip 

\noindent\textbf{Keywords:} Symmetric functions; partition division; $k$-Yamanouchi tableaux; positivity.
\medskip 
		
\noindent\textbf{AMS Classification 2020:} 05E05; 05A17; 05A19.
\end{abstract}

\maketitle

\section{Introduction}
Stanley introduced an operator $T_k$ on symmetric functions in his study of flag-symmetric posets~\cite{StanleyFlagSymmetric96}. For a flag-symmetric poset $P$ of rank $kn$, the operator sends the chain enumerator of $P$ to that of its rank-selected subposet on ranks $0,k,\ldots,kn$, with the ranks divided by $k$. We denote this operator by $\rowMap_k$ and call it the partition division map.

Stanley proved that $T_k$ preserves Schur positivity
\cite[Theorem~2.4]{StanleyFlagSymmetric96}. His proof gives the Schur coefficients in terms of Littlewood--Richardson coefficients and expresses the image of a Schur function under the adjoint operator as a skew Schur function. We obtain a tableau formula for the Schur expansion of $\rowMap_k(\schur_{\lambda/\mu})$ using $k$-Yamanouchi tableaux.

The map $\rowMap_k$ is also related to stretched Kostka coefficients. For $\lambda,\mu\vdash n$, the coefficient of $\monomial_\mu$ in $\rowMap_k(\schur_{k\lambda})$ is $K_{k\lambda,k\mu}$. For fixed $\lambda$ and $\mu$, these coefficients count lattice points in dilations of a Gelfand--Tsetlin polytope and are polynomial in $k$. McAllister~\cite{Mcallister2008} determined the degree of this polynomial; see also \cite{KingTolluToumazet2004}. In Section~\ref{Preliminaries}, we give an interpretation using tableaux with fractional weights.

Gelfand--Tsetlin patterns also arise in continuous analogues of tableau combinatorics. O'Connell~\cite{OConnell2014} uses integrals over real Gelfand--Tsetlin patterns to study Whittaker functions. Tao~\cite{Tao2017} studies continuous analogues of Schur and skew Schur polynomials, while Prasad~\cite{Prasad2018} extends tableau and plactic combinatorics to timed words. In our setting, tableaux with fractional weights give a discrete interpretation of the stretched Kostka coefficients appearing in $\rowMap_k(\schur_{k\lambda})$.

The LLT polynomials introduced by Lascoux, Leclerc, and
Thibon~\cite{LLT1997} are Schur-positive, as proved by Grojnowski and Haiman~\cite{GrojnowskiHaiman2006}. For the chromatic symmetric functions introduced by Stanley~\cite{Stanley1995},
Gasharov~\cite{Gasharov1996} proved Schur positivity for
incomparability graphs of $(3+1)$-free posets. The Stanley--Stembridge conjecture asked for $\elementaryE$-positivity in the same setting; Guay-Paquet~\cite{GuayPaquet2013} reduced it to unit interval graphs, and Hikita~\cite{Hikita2026} proved it.

The map $\rowMap_k$ is also closely related to classical operations on symmetric functions. We also compare $\rowMap_k$ with the Verschiebung operator in Section~\ref{Application}. The latter is adjoint to the Adams operator with respect to the Hall inner product~\cite{Macdonald1995,StanleyEC2}. The two operators are distinct for $k>1$, since $\rowMap_k(\powersum_k)=\powersum_1,
\versch_k(\powersum_k)=k\powersum_1$.

The first version of this paper conjectured that $\rowMap_k$ preserves
$\elementaryE$-positivity. This conjecture has since been proved in our paper~\cite{AlexanderssonDai2026Epositive}.
We refer to that paper for the proof and the weighted formula for the
individual elementary coefficients. The tableau and cyclic-word formulas
developed here do not use that proof.

The organization of this paper is as follows. In
Section~\ref{Preliminaries}, we review the preliminaries and introduce
$\rowMap_k$ and $\colMap_k$. In Section~\ref{rowmap_kOns}, we give the
fundamental quasisymmetric expansion of $\rowMap_k(\schur_\mu)$
(Theorem~\ref{thm:F-expansion}) and its Schur expansion via
$k$-Yamanouchi tableaux (Theorem~\ref{thm:schurExpansion}). In
Section~\ref{rowmap_kOne}, we study $\rowMap_k(\elementaryE_\mu)$.
We determine its Schur expansion (Theorem~\ref{thm:emuSchurExpansion})
and relate it to work of Amdeberhan, Shareshian, and Stanley. We also
obtain a power-sum expansion of $\omega(\rowMap_k(\elementaryE_\mu))$
(Theorem~\ref{thm:mpos} and Corollary~\ref{coro:p-positive}) and a
combinatorial formula for sums of elementary coefficients over
partitions of a fixed length (Theorem~\ref{thm:fixed-block-tableaux}).
Finally, in Section~\ref{Application}, we present applications and
compare the division maps with classical operators.

\section{Preliminaries}\label{Preliminaries}
We recall some notation for partitions and tableaux. A \defin{partition} $\lambda = (\lambda_1,\lambda_2,\dots,\lambda_\ell) \vdash n$ is a weakly decreasing sequence of positive integers such that $\sum_i \lambda_i = n$. Its length is $\ell(\lambda)=\ell$. We identify $\lambda$ with its Young diagram. For example,
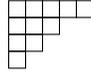
\[
\ytableausetup{boxsize=0.7em}
\ydiagram{5,3,2,1} 
\]
is the Young diagram of $\lambda=(5,3,2,1)$ in English notation.

For each $i \ge 1$, let $m_i = m_i(\lambda)$ denote the number of times $i$ appears in $\lambda$, and
\[
\defin{z_\lambda} \coloneqq \prod_{i \ge 1} i^{m_i} m_i!.
\]

For partitions $\lambda,\mu \vdash n$, we write $\lambda \unrhd \mu$ if 
\[
\sum_{i=1}^j \lambda_i \ge \sum_{i=1}^j \mu_i
\quad \text{for all } j \ge 1,
\]
and call this the \defin{dominance order}. We write $\mu\lhd\lambda$ if $\lambda\unrhd\mu$ and $\lambda\ne\mu$.

A \defin{semistandard Young tableau} ($\SSYT$) of shape $\lambda$ is a 
filling of the diagram with positive integers that are weakly 
increasing along rows and strictly increasing down columns.

Let $\Lambda = \bigoplus_{n \ge 0} \Lambda^n$ denote the ring of symmetric functions over $\mathbb{Q}$. 

\begin{definition}[Standard bases of $\Lambda$]
For a partition $\lambda = (\lambda_1, \lambda_2, \dots) \vdash n$, the \defin{monomial symmetric function} $\monomial_\lambda$ is defined as
\[
\defin{\monomial_\lambda(\xvec)} \coloneqq \sum_{\alpha} x_1^{\alpha_1} x_2^{\alpha_2} \cdots,
\]
where the sum is over all distinct permutations $\alpha$ of $\lambda$.

For $k \ge 0$, the \defin{complete homogeneous symmetric function} $\completeH_k$ and the \defin{elementary symmetric function} $\elementaryE_k$ are defined by
\[
\defin{\completeH_k(\xvec)} \coloneqq \sum_{i_1 \le \cdots \le i_k} x_{i_1} \cdots x_{i_k},
\quad
\defin{\elementaryE_k(\xvec)} \coloneqq \sum_{i_1 < \cdots < i_k} x_{i_1} \cdots x_{i_k}.
\]
We use $\completeH_0=\elementaryE_0=1$ and $\completeH_r=\elementaryE_r=0$ for $r<0$.
For a partition $\lambda$, let $\completeH_\lambda \coloneqq \completeH_{\lambda_1} \completeH_{\lambda_2} \cdots$ and $\elementaryE_\lambda \coloneqq \elementaryE_{\lambda_1} \elementaryE_{\lambda_2} \cdots$.

For $k \ge 1$, the \defin{power sum symmetric function} $\powersum_k$ is defined by
\[
\defin{\powersum_k(\xvec)} \coloneqq \sum_i x_i^k,
\]
and for a partition $\lambda$, let $\powersum_\lambda \coloneqq \powersum_{\lambda_1} \powersum_{\lambda_2} \cdots$.

Recall that the \defin{Schur function} $\schur_\lambda$ is defined as
\begin{equation}\label{eq:schurDef}
\defin{ \schur_{\lambda}(\xvec) } \coloneqq \sum_{T \in \SSYT(\lambda)} \xvec^T =
 \sum_{\mu} K_{\lambda \mu} \monomial_\mu(\xvec),
\end{equation}
where $K_{\lambda \mu}$ are the Kostka coefficients.
\end{definition}

For a skew shape $\lambda/\mu$, the \defin{skew Schur function} is defined by
\begin{equation}\label{eq:skewSchurDef}
\defin{\schur_{\lambda/\mu}(\xvec)} \coloneqq \sum_{T \in \SSYT(\lambda/\mu)} \xvec^T,
\end{equation}
where $\SSYT(\lambda/\mu)$ denotes the set of semistandard Young tableaux of skew shape $\lambda/\mu$.

A symmetric function is called \defin{$f$-positive} if its expansion in
the basis $\{f_\lambda\}$ has nonnegative rational coefficients.
The elementary and Schur coefficients counted by tableaux below are
integers. Power-sum positivity does not require integrality.

The \defin{Hall inner product} $\langle\cdot, \cdot \rangle$ on $\Lambda$ is defined by 
\[
\langle\monomial_\lambda, \completeH_\mu\rangle\coloneqq\delta_{\lambda\mu}.
\]
The Schur functions are orthonormal: $\langle\schur_\lambda, \schur_\mu\rangle=\delta_{\lambda\mu}.$

The \defin{reading word} of a tableau is obtained by reading its entries row by row from left to right, starting from the bottom row and proceeding to the top.

A word $w = w_1 w_2 \cdots w_N$ in positive integers is called \defin{Yamanouchi} if, in every prefix of the word, the number of occurrences of $i$ is at least the number of occurrences of $i+1$ for all $i \ge 1$. A semistandard Young tableau is called a \defin{Yamanouchi tableau} if its reverse reading word is Yamanouchi. The Littlewood--Richardson rule relates these tableaux to products and skew Schur functions.
\begin{theorem}[Littlewood--Richardson rule, see e.g.~\cite{StanleyEC2}]\label{thm:LRRule}
We have
\[
 \schur_{\lambda /\mu} = \sum_{\theta} c^{\lambda}_{\mu,\theta} \schur_\theta,
\text{ and }
 \schur_{\mu} \cdot \schur_\theta = \sum_{\lambda} c^{\lambda}_{\mu,\theta} \schur_\lambda,
\]
where $c^{\lambda}_{\mu,\theta}$ is the number of Yamanouchi tableaux of shape $\lambda/\mu$ and content $\theta$.
\end{theorem}

For further background on symmetric functions and tableaux, see~\cite{Sagan2001,Fulton1997}.

\subsection{Two maps on symmetric functions}

Let $\lambda = (\lambda_1,\dotsc,\lambda_\ell)$ be a partition and $k$ a positive integer. Let $k\lambda$ denote $(k\lambda_1,\dotsc,k\lambda_\ell)$ and let $\lambda^k$ denote 
\[
(\lambda_1^k,...,\lambda_\ell^k)=(\underbrace{\lambda_1,\dotsc,\lambda_1}_k,
\dotsc,
\underbrace{\lambda_\ell,\dotsc,\lambda_\ell}_k
).
\]
Then $\lambda^k = (k\lambda ')'$, where $\lambda'$ is the transpose of $\lambda$.

\begin{example}
Suppose $\lambda=(4,2,1)$. Then the Young diagrams of $3\lambda$ and $\lambda^3$ are 
\[
\ytableausetup{boxsize=0.7em}
\ydiagram{12,6,3} 
\qquad 
\text{ and }
\qquad 
\ydiagram{4,4,4,2,2,2,1,1,1}.
\]
\end{example}

If every part of $\mu$ is divisible by $k$, let $\mu/k$ be the partition obtained by dividing each part by $k$. If every part of $\mu'$ is divisible by $k$, let $\mu^{1/k}=(\mu'/k)'$.

Let \defin{$\rowMap_k$} and \defin{$\colMap_k$} be ``\emph{partition division maps}'' on 
the monomial basis of the symmetric functions as follows:
\begin{align}
 \defin{\rowMap_k\left(\monomial_{\mu}(\xvec)\right)} &\coloneqq 
 \begin{cases}
  \monomial_{\frac{1}{k} \mu}(\xvec)  & \text{ if all entries of $\mu$ are divisible by $k$} \\
  0 & \text{ otherwise.}
 \end{cases}\label{eq:rowmapDef}
\\
 \defin{ \colMap_k\left(\monomial_{\mu}(\xvec)\right) } &\coloneqq 
 \begin{cases}
  \monomial_{\mu^{1/k}}(\xvec)  & \text{ if all entries of $\mu'$ are divisible by $k$} \\
  0 & \text{ otherwise.}
 \end{cases}\label{eq:colmapDef}
\end{align}

Note that $\rowMap_k$ and $\colMap_k$ send symmetric functions of degree $nk$
to symmetric functions of degree $n$ (or $0$).

For $k>1$, the maps $\rowMap_k$ and $\colMap_k$ are linear but are not ring homomorphisms. Their adjoints are ring homomorphisms for every $k\ge1$.

\begin{lemma}
Let $\hat{\omega}$ be the involution on the space of symmetric functions that 
sends $\monomial_\mu$ to $\monomial_{\mu'}$. Then 
\[
\colMap_k = \hat{\omega} \circ \rowMap_k \circ\; \hat{\omega}.
\]
\end{lemma}

\begin{proposition}\label{prop:adjMap}
The adjoints of $\rowMap_k$ and $\colMap_k$ with respect to the Hall inner product are \defin{$\rowAdj{k}$} and \defin{$\colAdj{k}$}, respectively, satisfying
\begin{equation}
  \rowAdj{k}( \completeH_r ) = \completeH_{kr},
\quad 
\text{ and }
\quad 
  \colAdj{k}( \completeH_r ) = \completeH_{r}^k.
\end{equation}
\end{proposition}
\begin{proof}
For $\alpha\vdash kn$ and $\beta\vdash n$, the definitions and Hall duality give
\[
\langle\rowMap_k(\monomial_\alpha),\completeH_\beta\rangle
=\delta_{\alpha,k\beta},\qquad
\langle\colMap_k(\monomial_\alpha),\completeH_\beta\rangle
=\delta_{\alpha,\beta^k}.
\]
Thus $\rowAdj{k}(\completeH_\beta)=\completeH_{k\beta}$ and $\colAdj{k}(\completeH_\beta)=\completeH_{\beta^k}$. These formulas also show that both adjoints are ring homomorphisms.
\end{proof}
For a composition $\alpha=(\alpha_1,\ldots,\alpha_\ell)$, let
\[
 \qmonomial_\alpha(\xvec)=
 \sum_{i_1<\cdots<i_\ell}x_{i_1}^{\alpha_1}\cdots x_{i_\ell}^{\alpha_\ell}.
\]
These are the monomial quasisymmetric functions. We extend the two maps by
\[
\rowMap_k(M_\alpha)=\begin{cases}M_{\alpha/k},&k\mid\alpha_i\text{ for all }i,\\0,&\text{otherwise,}\end{cases}
\qquad
\colMap_k(M_\alpha)=\begin{cases}M_\beta,&\alpha=\beta^k,\\0,&\text{otherwise,}\end{cases}
\]
where $\beta^k$ repeats each part consecutively $k$ times. Indeed, $\monomial_\lambda$ is the sum of $\qmonomial_\alpha$ over
the distinct rearrangements $\alpha$ of $\lambda$. The displayed rules
therefore restrict to the symmetric-function maps above. We also use $\rowMap_k$ on polynomials: it sends $x_1^{a_1}\cdots x_N^{a_N}$ to $x_1^{a_1/k}\cdots x_N^{a_N/k}$ if every $a_i$ is divisible by $k$, and to zero otherwise.

\subsection{The Schur functions under the two division maps}

Let $k$ be a positive integer and let $\lambda \vdash nk$.
By definition, 
\begin{equation}
\rowMap_k(\schur_\lambda(\xvec)) =
\sum_{\mu \vdash n} K_{\lambda,k\mu} \monomial_\mu(\xvec),
\end{equation}
where $K_{\lambda,k\mu}$ is a Kostka coefficient as in \eqref{eq:schurDef}.
Similarly, 
\begin{equation}
\colMap_k(\schur_\lambda(\xvec)) = \sum_{\mu \vdash n} K_{\lambda,\mu^k} \monomial_\mu(\xvec).
\end{equation}

In particular, if $\lambda \vdash n$, we have
\begin{equation}\label{eq:rowColMapSpecialCase}
\rowMap_k(\schur_{k\lambda}(\xvec)) = \sum_{\mu} K_{k\lambda,k\mu} \monomial_\mu(\xvec),
\qquad 
\colMap_k(\schur_{\lambda^k}(\xvec)) = \sum_{\mu} K_{\lambda^k,\mu^k} \monomial_\mu(\xvec).
\end{equation}

The numbers $K_{k\lambda,k\mu}$ are called \defin{stretched Kostka coefficients}; see \cite{KingTolluToumazet2004,Mcallister2008,Rassart2004}.

We can interpret the coefficients $K_{k\lambda,k\mu}$ and 
 $K_{\lambda^k,\mu^k}$ as the number of ways to
fill the diagram $\lambda$ with $\mu_i$ boxes equal to $i$, for $i \geq 1$
but we allow boxes to contain fractional weights, which are multiples of $\frac{1}{k}$. 
To obtain $K_{k\lambda,k\mu}$, subdivide each column into $k$
fractional columns, which must still be strictly increasing,
while $K_{\lambda^k,\mu^k}$ is obtained when 
each row instead is sub-divided into $k$ weakly increasing fractional rows.

\begin{remark}
For $\lambda\vdash kn$, the function $\rowMap_k(\schur_\lambda)$ is zero if and only if no partition $\mu\vdash n$ satisfies $\lambda\unrhd k\mu$.
\end{remark}

\begin{example}
Each semistandard tableau in $\SSYT(\lambda,\mu)$
injects naturally into the set of tableaux in $\SSYT(k\lambda,k\mu)$,
and  $\SSYT(\lambda^k,\mu^k)$, respectively.
Here, $\lambda=331$ and $k=2$:

\begin{figure}[htbp]
    \centering
    \begin{minipage}{0.25\textwidth}
        \centering
        \ytableausetup{boxsize=1.5em}
        \ytableaushort{112,234,3}
    \end{minipage}
    \hfill
    \begin{minipage}{0.35\textwidth}
        \centering
        \begin{tikzpicture}[scale=0.7]
          \draw (0,0) -- (1,0);  
          \draw (0,1) -- (3,1);  
          \draw (0,2) -- (3,2);  
          \draw (0,3) -- (3,3);  
          \draw (0,0) -- (0,3);  
          \draw[dashed] (0.5,0) -- (0.5,3);  
          \draw (1,0) -- (1,3);  
          \draw[dashed] (1.5,1) -- (1.5,3);
          \draw (2,1) -- (2,3);  
          \draw[dashed] (2.5,1) -- (2.5,3);
          \draw (3,1) -- (3,3);  
          \node at (0.25, 0.5) {3};
          \node at (0.75, 0.5) {3};
          
          \node at (0.25, 1.5) {2};
          \node at (0.75, 1.5) {2};
          \node at (1.25, 1.5) {3};
          \node at (1.75, 1.5) {3};
          \node at (2.25, 1.5) {4};
          \node at (2.75, 1.5) {4};
          
          \node at (0.25, 2.5) {1};
          \node at (0.75, 2.5) {1};
          \node at (1.25, 2.5) {1};
          \node at (1.75, 2.5) {1};
          \node at (2.25, 2.5) {2};
          \node at (2.75, 2.5) {2};
        \end{tikzpicture}
    \end{minipage}
    \hfill
    \begin{minipage}{0.35\textwidth}
        \centering
        \begin{tikzpicture}[scale=0.7]
          \draw (0,0) -- (1,0);  
          \draw[dashed]  (0,0.5) -- (1,0.5);
          \draw (0,1) -- (3,1);  
          \draw[dashed] (0,1.5) -- (3,1.5); 
          \draw (0,2) -- (3,2);  
          \draw[dashed] (0,2.5) -- (3,2.5);
          \draw (0,3) -- (3,3);  
          \draw (0,0) -- (0,3);  
          \draw (1,0) -- (1,3);  
          \draw (2,1) -- (2,3);  
          \draw (3,1) -- (3,3);  
          \node at (0.5, 0.25) {6};
          \node at (0.5, 0.75) {5};
          
          \node at (0.5, 1.25) {4};
          \node at (0.5, 1.75) {3};
          \node at (1.5, 1.25) {6};
          \node at (1.5, 1.75) {5};
          \node at (2.5, 1.25) {8};
          \node at (2.5, 1.75) {7};
          
          \node at (0.5, 2.25) {2};
          \node at (0.5, 2.75) {1};
          \node at (1.5, 2.25) {2};
          \node at (1.5, 2.75) {1};
          \node at (2.5, 2.25) {4};
          \node at (2.5, 2.75) {3};
        \end{tikzpicture}
    \end{minipage}
\end{figure}
\end{example}

\begin{lemma}
For a fixed positive integer $k$, the set of functions 
$\{\rowMap_k(\schur_{k\lambda}) :\lambda \vdash n\}$ is linearly independent and forms a basis for the symmetric function space $\Lambda^n$.
The same statement holds for $\{\colMap_k(\schur_{\lambda^k}) :\lambda \vdash n\}$.
\end{lemma}
\begin{proof} Since $\dim \Lambda^n = p(n)$, where $p(n)$ is the number of partitions of $n$, and there are $p(n)$ such $\rowMap_k(\schur_{k\lambda})$, it suffices to show linear independence.

Consider the transition matrix $M =(K_{k\lambda, k\nu})_{\lambda, \nu \vdash n}$, where rows and columns are indexed by partitions of $n$. This matrix relates the set $\{\rowMap_k(\schur_{k\lambda})\}$ to the monomial basis $\{ \monomial_\nu \}$ via
\[
(\rowMap_k(\schur_{k\lambda}))_{\lambda \vdash n} = M (\monomial_\nu)_{\nu \vdash n}.
\]
If $\lambda \unrhd  \mu$, then for each $i$,
\[
k\lambda_1+ \dotsb +k\lambda_i=k (\lambda_1+\cdots+\lambda_i)\geq k (\mu_1+\dotsb+\mu_i)=k\mu_1+\cdots + k\mu_i,
\]
so $k\lambda \unrhd k\mu$. 
Therefore, by the property of Kostka coefficients, $K_{k\lambda, k\nu}=0$ unless $k\lambda \unrhd k\nu$, 
which holds iff $\lambda \unrhd \nu$. The diagonal entries $K_{k\lambda,k\lambda}$ are all $1$. The matrix is therefore unitriangular, and its rows form a basis.

For the column map, $\colMap_k(\schur_{\lambda^k})=\sum_{\nu\vdash n}K_{\lambda^k,\nu^k}\monomial_\nu$. We have $K_{\lambda^k,\nu^k}=0$ unless $\lambda^k\unrhd\nu^k$. Since repeating each part $k$ times preserves the dominance order in both directions, $K_{\lambda^k,\nu^k}=0$ unless $\lambda\unrhd\nu$. Also, $K_{\lambda^k,\lambda^k}=1$. The transition matrix is therefore unitriangular.
\end{proof}

\begin{remark}
Let $\alpha\vdash kn$, and let
$H=(\mathfrak S_k)^n\subseteq\mathfrak S_{kn}$ permute entries
within each of $n$ consecutive blocks of size $k$. Permuting these blocks gives an
$\mathfrak S_n$-action on the invariant space
\[
 U_{\alpha,k}\coloneqq(S^\alpha)^H,
\]
where $S^\alpha$ is the Specht module of shape $\alpha$.
Its dimension agrees with the dimension prescribed by
$\rowMap_k(\schur_\alpha)$, since
\[
 \dim U_{\alpha,k}
 =\langle\schur_\alpha,\completeH_k^n\rangle
 =K_{\alpha,(k^n)}
 =[\monomial_{1^n}]\rowMap_k(\schur_\alpha).
\]
The Frobenius characteristic, however, need not equal
$\rowMap_k(\schur_\alpha)$. The block action gives
\[
 \mathrm{ch}(U_{\alpha,k})
 =\sum_{\nu\vdash n}
   \langle\schur_\alpha,\schur_\nu[\completeH_k]\rangle
   \schur_\nu,
\]
where the square brackets denote plethysm.
Table~\ref{tab:invariant-module-comparison} compares these
functions for $k=n=3$. Partitions not listed have zero image
under $\rowMap_3$ and zero invariant space.

\begin{table}[htbp]
\centering
\caption{The image $\rowMap_3(\schur_\alpha)$ and the
Frobenius characteristic of $U_{\alpha,3}$ for $\alpha\vdash9$.}
\label{tab:invariant-module-comparison}
\small
\begin{tabular}{@{}llll@{}}
\toprule
$\alpha$ & Monomial expansion & Schur expansion
& $\mathrm{ch}(U_{\alpha,3})$ \\
\midrule
$(9)$ & $\monomial_3+\monomial_{21}+\monomial_{111}$
& $\schur_3$ & $\schur_3$ \\
$(8,1)$ & $\monomial_{21}+2\monomial_{111}$
& $\schur_{21}$ & $\schur_{21}$ \\
$(7,2)$ & $\monomial_{21}+3\monomial_{111}$
& $\schur_{21}+\schur_{111}$ & $\schur_3+\schur_{21}$ \\
$(7,1,1)$ & $\monomial_{111}$
& $\schur_{111}$ & $\schur_{111}$ \\
$(6,3)$ & $\monomial_{21}+4\monomial_{111}$
& $\schur_{21}+2\schur_{111}$
& $\schur_3+\schur_{21}+\schur_{111}$ \\
$(6,2,1)$ & $2\monomial_{111}$
& $2\schur_{111}$ & $\schur_{21}$ \\
$(5,4)$ & $2\monomial_{111}$
& $2\schur_{111}$ & $\schur_{21}$ \\
$(5,3,1)$ & $3\monomial_{111}$
& $3\schur_{111}$ & $\schur_{21}+\schur_{111}$ \\
$(5,2,2)$ & $\monomial_{111}$
& $\schur_{111}$ & $\schur_3$ \\
$(4,4,1)$ & $\monomial_{111}$
& $\schur_{111}$ & $\schur_3$ \\
$(4,3,2)$ & $2\monomial_{111}$
& $2\schur_{111}$ & $\schur_{21}$ \\
$(3,3,3)$ & $\monomial_{111}$
& $\schur_{111}$ & $\schur_{111}$ \\
\bottomrule
\end{tabular}
\end{table}
\end{remark}

\section{The expansion of \texorpdfstring{$\rowMap_k(\schur_{\mu})$}{rowDiv of Schur}}\label{rowmap_kOns}
\subsection{Fundamental quasisymmetric expansions}
We use Gessel's fundamental quasisymmetric basis~\cite{Gessel1984} to expand $\rowMap_k (\schur_\lambda(\xvec))$.
\begin{definition}
Let $\alpha = (\alpha_1,\alpha_2,\dots,\alpha_\ell)$ be a composition of $n$. The \defin{fundamental quasisymmetric function} $F_\alpha$ is defined by
\[
\defin{F_\alpha(\xvec)} \coloneqq \sum_{\substack{i_1 \le i_2 \le \cdots \le i_n \\ i_j < i_{j+1}, j \in S_\alpha}} x_{i_1} x_{i_2} \cdots x_{i_n},
\]
where $S_\alpha = \{\alpha_1, \alpha_1+\alpha_2, \dots, \alpha_1+\cdots+\alpha_{\ell-1}\}$.
\end{definition}
\begin{lemma}
The image of a fundamental quasisymmetric function under $\rowMap_k$ is again 
a fundamental quasisymmetric function, or $0$:
\[
\rowMap_k\left( \gessel_\alpha(\xvec) \right) = 
\begin{cases}
\gessel_{\frac{1}{k}\alpha}(\xvec) & \text{if all entries of $\alpha$ are divisible by $k$} \\
0 & \text{otherwise}.
\end{cases}
\]
\end{lemma}
\begin{proof}
This follows directly from the definition.
\end{proof}

\begin{remark}
The image of $\gessel_\alpha(\xvec)$ under $\colMap_k$ is more complicated.
For example,
\[
 \colMap_2 \left( \gessel_{42}(\xvec) \right) = 
 \qmonomial_{12} + \qmonomial_{21} + \qmonomial_{111}
 =\gessel_{12}+\gessel_{21} - \gessel_{111}.
\]
\end{remark}
\subsection{Expanding \texorpdfstring{$\rowMap_k \schur_{\lambda}(\xvec)$}{rowDiv of Schur function} in the fundamental quasisymmetric basis}

For a standard Young tableau $T$ of size $n$, 
the \defin{descent set} $\DES(T)$ is the set of 
entries $i \in \{1,\dotsc,n-1\}$ such that $i+1$ appears 
in a row strictly below $i$ in $T$ (see~\cite[Section~7.19]{StanleyEC2}).

To standardize a semistandard tableau, replace its entries by consecutive integers. First replace the entries equal to $1$ from left to right, then those equal to $2$, and so on.

Suppose $\lambda \vdash kn$ where $k \in \setN$.
Let $\defin{\SYT_k(\lambda)}$ be the set of standard Young tableaux such that 
every element in the descent set is a multiple of $k$.
The proof below shows that standardization gives a bijection from $\SSYT(\lambda,1^k 2^k \dotsb n^k)$ to $\SYT_k(\lambda)$.

For $S\subseteq\{1,\ldots,n-1\}$, write $F_{n,S}=F_\alpha$, where $S$ is the set of partial sums of the composition $\alpha$ of $n$.

\begin{theorem}\label{thm:F-expansion}
Suppose $\lambda \vdash kn$. We have
\begin{equation}\label{eq:rowGessel}
\rowMap_k (\schur_{\lambda}(\xvec) )= \sum_{T \in \SYT_{k}(\lambda)} \gessel_{n, \frac{1}{k}\DES(T)}(\xvec)
\end{equation}
Here $\frac{1}{k}\DES(T)$ denotes the set $\left\{\frac{j}{k} : j \in \DES(T)\right\}$.
\end{theorem}
\begin{proof}
By definition,
\[
\rowMap_k (\schur_{\lambda}(\xvec) )= \sum_\beta
\sum_{T' \in \SSYT(\lambda, k\beta )} \xvec^{\beta},
\]
where the sum is over all weak compositions $\beta$.
The cells of $T'$ with any one entry form a horizontal strip. The integers assigned to these cells by standardization have no descents between them. Each entry occurs a multiple of $k$ times, so all descents of the standardized tableau occur at multiples of $k$. Thus this tableau belongs to $\SYT_k(\lambda)$.

Conversely, let $T\in\SYT_k(\lambda)$. Replacing the
entries $(j-1)k+1,\ldots,jk$ by $j$ gives a semistandard
tableau, since $T$ has no descents within these intervals.
This construction is inverse to standardization for
content $1^k\cdots n^k$.

For fixed $T$, every tableau with standardization $T$ and content divisible by $k$ is obtained by replacing block $j$ by $i_j$, where
\[
 i_1\le\cdots\le i_n,\qquad
 i_j<i_{j+1}\quad\text{if }kj\in\DES(T).
\]
Its contribution after division is $x_{i_1}\cdots x_{i_n}$. Summing these monomials gives $F_{n,\DES(T)/k}$, proving the formula.
\end{proof}

\begin{figure}[htbp]
\centering
\[
\ytableausetup{boxsize=1.1em,aligntableaux=center}
\begin{ytableau}
1&1&1&1&1&1&2&2&*(pastelblue)3\\
2&2&2&2&*(pastelblue)3&4\\
*(pastelblue)3&*(pastelblue)3&*(pastelblue)3&*(pastelblue)3&4\\
4&5&5&5
\end{ytableau}
\quad\xrightarrow{\ \std\ }\quad
\begin{ytableau}
1&2&3&4&5&*(lightgray)6&11&*(lightgray)12&*(lightgray)18\\
7&8&9&10&17&*(lightgray)21\\
13&14&15&16&20\\
19&22&23&24
\end{ytableau}
\]
\caption{Standardization for $k=3$, $\lambda=(9,6,5,4)$,
and $\beta=(2,2,2,1,1)$. The tableau on the left has content
$k\beta=(6,6,6,3,3)$; its entries equal to $3$ form the blue
horizontal strip. The shaded entries on the right are the
descents $6,12,18,21$, all divisible by $3$.}
\label{fig:standardization}
\end{figure}

\subsection{Schur expansions}
We now turn to the Schur expansion of $\rowMap_k(\schur_{\lambda/\mu})$. We begin with a formula for the adjoint operator.

\begin{proposition}\label{prop:rowMapAdjSchurExp}
Let $\lambda/\mu$ be a skew shape of size $n$ and $\rho \coloneqq (\ell-1,\ell-2,\dots,1,0)$,
where $\ell$ is the number of parts of $\lambda$. Then
\begin{equation}\label{eq:rowMapAdjSchurExp}
\rowAdj{k}\left(\schur_{\lambda/\mu}(\xvec)\right) 
= \sum_{\nu \vdash kn} c^{k\lambda + (k-1)\rho}_{k\mu + (k-1)\rho, \nu} \, \schur_\nu(\xvec).
\end{equation}
In particular,
\begin{equation}
\rowAdj{k}(\schur_{\lambda}(\xvec)) = 
\sum_{\nu \vdash kn}  
c^{k\lambda + (k-1)\rho}_{(k-1)\rho, \nu}  \, \schur_{\nu}(\xvec).
\end{equation}
\end{proposition}
\begin{proof}
We only prove \eqref{eq:rowMapAdjSchurExp} since the second identity corresponds to $\mu=\emptyset$.
By the Jacobi--Trudi identity for skew Schur functions, we have
\begin{equation}\label{eq:rowMapAdj1}
\rowAdj{k}\left( \schur_{\lambda/\mu} \right) = 
\rowAdj{k}\left( \det\bigl(\completeH_{\lambda_i - \mu_j - i + j}\bigr)_{i,j} \right) 
= 
\det\bigl(\completeH_{k(\lambda_i - \mu_j - i + j)}\bigr)_{i,j} 
\end{equation}
where the second identity follows from Proposition~\ref{prop:adjMap}.
Set 
\[
\alpha = k\lambda + (k-1)\rho \quad \text{ and }  \quad \beta  = k\mu   + (k-1)\rho,
\]
then $\alpha_i - \beta_j - i + j = k(\lambda_i - \mu_j - i + j)$.
The Jacobi--Trudi identity for skew Schur functions gives 
\begin{equation}\label{eq:rowMapAdj2}
\schur_{\alpha/\beta}
  = \det\left(\completeH_{\alpha_i - \beta_j - i + j}\right)_{i,j}
  = \det\left(\completeH_{k(\lambda_i - \mu_j - i + j)}\right)_{i,j}.
\end{equation}
Comparing \eqref{eq:rowMapAdj1} and \eqref{eq:rowMapAdj2} gives
\[
\rowAdj{k}\left(\schur_{\lambda/\mu}\right)
  = \schur_{\,k\lambda + (k-1)\rho \,/\, k\mu + (k-1)\rho}.
\]
The Littlewood--Richardson rule applied to the 
right hand side then gives \eqref{eq:rowMapAdjSchurExp}.
\end{proof}

\begin{corollary}
Let $k\ge1$, $|\lambda|=kn$ in the first formula and $|\lambda/\mu|=kn$ in the second. Choose $L\ge n$, pad the output partitions with zeros to length $L$, and set $\rho=(L-1,\ldots,0)$. Then:
\begin{align}
\rowMap_k(\schur_{\lambda}(\xvec)) &= 
\sum_{\mu \vdash n}  c^{k\mu + (k-1)\rho}_{(k-1)\rho, \lambda}\schur_{\mu}(\xvec) \label{eq:rowMapOfSInS}
\\
\rowMap_k(\schur_{\lambda /\mu}(\xvec)) &= 
\sum_{\nu \vdash n}
\sum_{\theta \vdash |\lambda/\mu|}  
c^{k\nu + (k-1)\rho}_{(k-1)\rho, \theta} \cdot c^{\lambda}_{\mu,\theta} \, \schur_{\nu}(\xvec).  \label{eq:rowMapOfSkewSInS}
\end{align}
\end{corollary}
\begin{proof}
By Proposition~\ref{prop:rowMapAdjSchurExp} and the Littlewood--Richardson rule,
we have 
\[
\langle \schur_\lambda, \rowAdj{k}(\schur_\mu) \rangle = \langle \schur_\lambda,\schur_{{k\mu + (k-1)\rho / (k-1)\rho}}\rangle=\langle \schur_{(k-1)\rho}\schur_\lambda,\schur_{{k\mu + (k-1)\rho }}\rangle
=c^{k\mu + (k-1)\rho}_{(k-1)\rho, \lambda},
\]
which immediately gives \eqref{eq:rowMapOfSInS}.
The Littlewood--Richardson rule states 
\[
\schur_{\lambda /\mu} = \sum_{\theta} c^{\lambda}_{\mu,\theta} \, \schur_\theta,
\]
so applying $\rowMap_k$ on both sides and using \eqref{eq:rowMapOfSInS},
we get 
\[
\rowMap_k\left(\schur_{\lambda /\mu} \right) = 
\sum_{\theta} c^{\lambda}_{\mu,\theta} \left(
\sum_{\nu}
 c^{k\nu + (k-1)\rho}_{(k-1)\rho, \theta} \, \schur_\nu \right).
\]
This gives the expansion in \eqref{eq:rowMapOfSkewSInS}.
\end{proof}

\begin{remark}
    Stanley's proof gives both the Littlewood--Richardson coefficient formula \eqref{eq:rowMapOfSInS} and the skew Schur formula for the adjoint. See Stanley's theorem~\cite[Theorem~2.4]{StanleyFlagSymmetric96}; see also~\cite[Exercise~7.45]{StanleyEC2}. Our $k$-Yamanouchi tableaux give another combinatorial interpretation of these coefficients.
\end{remark}

\subsection{\texorpdfstring{$k$}{k}-Yamanouchi tableaux}
We define $k$-Yamanouchi tableaux and use them to interpret the Schur coefficients.

\begin{definition}\label{def:Yamanouchi}
A word $w$ is called a \defin{$k$-ballot word} if every prefix 
contains at most $k-1$ more $i+1$'s than $i$s for each $i\geq 1$.
A $1$-ballot word is called a \defin{ballot word}.
A semistandard Young tableau is called \defin{$k$-Yamanouchi} if its 
reverse reading word is a $k$-ballot word.
We let \defin{$\YSSYT_k(\lambda/\mu,\nu)$} be the set of $k$-Yamanouchi 
semistandard Young tableaux of skew shape $\lambda/\mu$ and content $\nu$. 
The number of $1$-Yamanouchi tableaux of given shape and type 
is a Littlewood--Richardson coefficient:
\[
 |\YSSYT_1(\lambda/\mu,\nu)| = c^{\lambda}_{\mu,\nu}.
\]
\end{definition}
\begin{example}
The tableau
\[
\ytableausetup{boxsize=0.95em}
\ytableaushort{1134,2245,35}
\]
has reading word $\mathtt{3522451134}$
and its reverse reading word is $\mathtt{4311542253}$.
This word is $2$-ballot, so the tableau is $2$-Yamanouchi.
\end{example}

The following proposition gives another characterization of $k$-ballot words.
\begin{proposition}\label{prop:ballotWord}
 A word $\wvec \in [n]^\ell$ is $k$-ballot if and only if 
\[
 1^{n (k-1)} \, 2^{(n-1)(k-1)} \, \dotsb \, n^{(k-1)}\, w_1 w_2 \dotsc w_\ell 
\]
is a ballot word. 
\end{proposition}
\begin{proof}
Let $\wvec=w_1\, w_2\, \dotsc \, w_\ell$ and define
\[
\uvec \coloneqq 1^{n (k-1)} \, 2^{(n-1)(k-1)} \, \dotsb \, n^{k-1}\, \wvec,
\]
where $N=(k-1)\binom{n+1}{2}$ is the length of the prefixed word. The prefix of length $N$ is ballot. For $1\le j\le n$, it contains $(n-j+1)(k-1)$ entries equal to $j$,
so whenever $1 \leq m \leq \ell$, we have 
\begin{align*}
  \lettercount_{j}(w_1\, w_2\, \dotsb \, w_{m}) &= \lettercount_{j}(u_1\, u_2\, \dotsb \, u_{N+m}) - (n-j+1)(k-1) , \\
\lettercount_{j+1}(w_1\, w_2\, \dotsb \, w_{m}) &= \lettercount_{j+1}(u_1\, u_2\, \dotsb \, u_{N+m}) - (n-j)(k-1) ,
\end{align*}
where $\lettercount_j(\cdot)$ denotes the number of $j$'s in the word.
\medskip 

By definition, $\wvec$ is $k$-ballot if and only if for all $m = 1,2,\dotsc,\ell$,
we have 
\[
 \lettercount_{j}(w_1\, w_2\, \dotsb \, w_m) + (k-1) \geq \lettercount_{j+1}(w_1\, w_2\, \dotsb \, w_m).
\]
This is equivalent to
\[
\lettercount_{j}(u_1\, \dotsb \, u_{N+m}) - (n-j+1)(k-1) + (k-1) \geq
 \lettercount_{j+1}(u_1\, \dotsb \, u_{N+m}) - (n-j)(k-1)
\]
which simplifies to
\[
\lettercount_{j}(u_1\, \dotsb \, u_{N+m}) \geq \lettercount_{j+1}(u_1\, \dotsb \, u_{N+m}).
\]
This is true for all $m$ exactly when $\uvec$ is a ballot word.
\end{proof}

\begin{proposition}\label{prop:kYamanouchiAsLRCoefficient}
The number of $k$-Yamanouchi semistandard tableaux 
in $\SSYT(\lambda,k\mu)$ is given by the Littlewood--Richardson coefficient 
\begin{equation}\label{eq:lrcoeff}
  c^{k\mu + (k-1)\rho}_{(k-1)\rho, \lambda}
\end{equation}
where $\rho$ is the staircase partition with the same number of parts as $\mu$.
\end{proposition}
\begin{proof}
The coefficient in \eqref{eq:lrcoeff} counts the number of semistandard tableaux
of skew shape $(k\mu + (k-1)\rho) / (k-1)\rho$ with content $\lambda$, 
and with reverse reading word being a ballot word; see for example \cite[Thm. 4.9.4]{Sagan2001}.

Given a $k$-Yamanouchi tableau $T \in \SSYT(\lambda,k\mu)$, 
we construct a tableau $T'$ with the above properties as follows. 
For each cell $(i,j)$ of $T$ with entry $T_{ij}$, we place an entry equal to $i$ in row $T_{ij}$ of $T'$. 
This map is called the \defin{companion map} in the literature;
see \cite[Def. 15]{KaliszewskiMorse2019}.

Let $\ell=\ell(\mu)$, and let $\rho=(\ell,\ell-1,\dots,1)$. This differs from $(\ell-1,\ldots,0)$ by $(1^\ell)$ and only translates the skew shape horizontally. Let $\beta=(k-1)\rho$, so that 
$\beta_p=(k-1)(\ell-p+1)$. We show the bijection
\[
\YSSYT_k(\lambda,k\mu)
\longrightarrow
\YSSYT_1\left(
(k\mu+\beta)/\beta,\lambda
\right).
\]

For each $p$, collect
all cells of $T$ containing $p$, ordered according to reverse reading
order:
\[
c_{p,1},c_{p,2},\dots,c_{p,k\mu_p}.
\]
If the cell $c_{p,t}$ lies in row $i$ of $T$, then place the entry $i$
in the $t$-th cell of row $p$ of $T'$, after the initial shift
$\beta_p$. Equivalently, 
$T'(p,\beta_p+t)=\operatorname{row}(c_{p,t})$.

Since $p$ appears $k\mu_p$ times in $T$, the $p$-th row of
$T'$ has $k\mu_p$ cells. Since this row is shifted by $\beta_p$, the
shape of $T'$ is $(k\mu+\beta)/\beta$.

Moreover, the number of entries equal to $i$ in $T'$ is exactly the number
of cells in row $i$ of $T$, namely $\lambda_i$. Hence $T'$ has content
$\lambda$.

We show that $T'$ is semistandard. Fix $p$.
The entries in row $p$ of $T'$ are the row indices of the cells of $T$
containing $p$, listed in reverse reading order. Therefore every row of $T'$ is weakly increasing. Also, we compare adjacent rows $p$ and $p+1$. Write $a_{p,t}=T'(p,\beta_p+t)$. Thus $a_{p,t}$ is the row index of the $t$-th occurrence of $p$ in the
reverse reading word of $T$. Since
\[
\beta_p-\beta_{p+1}=k-1,
\]
whenever the $t$-th skew cell in row $p$ and the
$(t+k-1)$-st skew cell in row $p+1$ both exist, they lie in the same column. We claim that
\[
a_{p,t}<a_{p+1,t+k-1}.
\]
Let $w$ be the reverse reading word of $T$. Since $T$ is
$k$-Yamanouchi, every prefix of $w$ satisfies
\[
 \#(p+1)\le \#p+(k-1).
\]
The prefix ending just before the $t$-th occurrence of $p$ contains $t-1$ occurrences of $p$. It therefore contains at most
$t+k-2$ occurrences of $p+1$. Thus the $(t+k-1)$-st occurrence of $p+1$ appears after the $t$-th occurrence of $p$.

The rows of $T$ are weakly increasing. Thus, within a row, $p+1$ comes before $p$ in reverse reading order. The two occurrences above must lie in different rows, with $p$ in the earlier row. Hence $a_{p,t}<a_{p+1,t+k-1}$, and $T'$ is semistandard.

To check the ballot condition for $T'$, take a prefix of its reverse reading word. Its entries correspond to a set $A$ of cells of $T$. The set $A$ contains all cells with entries less than some $p$ and an initial segment of the $p$-cells in the reverse of the order $c_{p,1},c_{p,2},\ldots$. In each row, these $p$-cells are selected from left to right. Thus $A$ is a Young diagram.

The number of $i$'s in the prefix equals the length of row $i$ of $A$. These lengths are weakly decreasing, so $\#i\ge\#(i+1)$ for every $i$. Thus $T'$ satisfies the ballot condition.

Conversely, let $T'$ be an LR tableau of this skew shape. Read its entries in reverse reading order. For each entry $i$ in row $p$ of $T'$, append a cell with entry $p$ to row $i$ of $T$. The ballot condition ensures that the shape remains a Young diagram after each insertion.

Since the rows of $T'$ are weakly increasing, the cells
with entry $p$ are added to $T$ in weakly decreasing
order of their row indices. They therefore form a
horizontal strip. As the values $p$ are processed in
increasing order, the resulting tableau $T$ is
semistandard, with shape $\lambda$ and content $k\mu$.

It remains to show that $T$ is $k$-Yamanouchi.
Order its $p$-cells as $c_{p,1},c_{p,2},\ldots$,
as above. Column strictness in $T'$ gives $a_{p,t}<a_{p+1,t+k-1}$ whenever the corresponding cells overlap.

Suppose that a prefix of the reverse reading word
of $T$ contains $k$ more $(p+1)$'s than $p$'s.
Take the first such prefix, and suppose it contains
$t-1$ occurrences of $p$. It must end with the
$(t+k-1)$-st occurrence of $p+1$.
Since $\mu_p\ge\mu_{p+1}$, the $t$-th occurrence
of $p$ exists, and the corresponding cells of
$T'$ overlap. Hence $a_{p,t}<a_{p+1,t+k-1}$.
Thus, the $t$-th occurrence of $p$ precedes the
$(t+k-1)$-st occurrence of $p+1$ in the reverse
reading word of $T$, a contradiction.
Hence, $T$ is $k$-Yamanouchi.

The two constructions are inverse. Therefore
\[
 \left|\YSSYT_k(\lambda,k\mu)\right|
 =
 \left|
 \YSSYT_1\left(
 (k\mu+(k-1)\rho)/(k-1)\rho,\lambda
 \right)
 \right|.
\]
\end{proof}

\begin{example}
Consider $k=2$, $\lambda = (6,3,1)$, and $\mu=(2,1,1,1)$.
There are four $k$-Yamanouchi semistandard tableaux in $\SSYT(\lambda,k\mu)$
and four Littlewood--Richardson tableaux of skew shape $8543/4321$ and content $631$.
The companion map gives the bijection below:
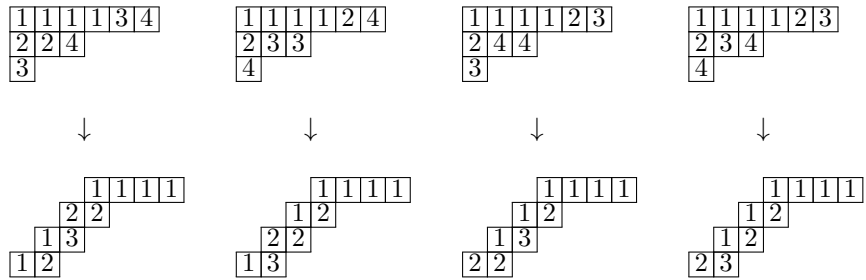
\[
\ytableausetup{boxsize=0.9em}
\begin{array}{cccc}
\ytableaushort{111134,224,3} & \ytableaushort{111124,233,4} 
& \ytableaushort{111123,244,3}&\ytableaushort{111123,234,4} \\
 \; & \; & \; & \;  \\
\downarrow & \downarrow & \downarrow & \downarrow \\
\; & \; & \; & \; \\
\ytableaushort{{\none}{\none}{\none}{\none}1111,{\none}{\none}{\none}22,{\none}{\none}13,{\none}12} &
\ytableaushort{{\none}{\none}{\none}{\none}1111,{\none}{\none}{\none}12,{\none}{\none}22,{\none}13} &
\ytableaushort{{\none}{\none}{\none}{\none}1111,{\none}{\none}{\none}12,{\none}{\none}13,{\none}22} & \ytableaushort{{\none}{\none}{\none}{\none}1111,{\none}{\none}{\none}12,{\none}{\none}12,{\none}23}
\end{array}
\]
\end{example}

\begin{lemma}\label{lem:skewYSSYTandLR}
For a skew shape $\lambda/\mu$ of size $nk$, and $ \nu \vdash n$, we have
\begin{equation}
|\YSSYT_k(\lambda/\mu, k\nu) | = 
\sum_{\theta \vdash nk}  
 c^{\lambda}_{\mu,\theta} \cdot  |\YSSYT_k(\theta, k\nu)|.
\end{equation}
\end{lemma}
\begin{proof}
We use the classical bijection given by \defin{jeu de taquin rectification},
see \cite[Chapter~5]{Fulton1997} for the rectification bijection.
For a skew shape $\lambda/\mu$ and content $\alpha$, the map
\[
\Phi: \SSYT(\lambda/\mu, \alpha) \longrightarrow \bigsqcup_{\theta \vdash |\alpha|} \mathrm{LR}(\lambda/\mu, \theta) \times \SSYT(\theta, \alpha)
\]
defined by $T \mapsto (R, S)$, 
where $S = \mathrm{rect}(T)$ and $R$ is the recording tableau, is a bijection. 
Here, $\mathrm{LR}(\lambda/\mu, \theta)$ denotes 
the set of Littlewood--Richardson tableaux of shape $\lambda/\mu$ and content $\theta$, 
the cardinality of which is $c^{\lambda}_{\mu,\theta}$.

It remains to show that this bijection restricts to the set of $k$-Yamanouchi tableaux. 
The reading words of $T$ and $S$ are Knuth equivalent. For fixed $i$, delete all letters other than $i$ and $i+1$, then cancel adjacent pairs $(i+1)i$ until none remain. The number of uncancelled $i+1$ letters is the maximum, over suffixes of the reading word including the empty suffix, of the excess of $i+1$ over $i$. A direct check of the two Knuth relations shows that the reduced two-letter word is unchanged, so this number is invariant under Knuth equivalence. This maximum is also the maximum excess over prefixes of the reverse reading word. Hence $T$ is $k$-Yamanouchi if and only if $S$ is $k$-Yamanouchi.
Summing over all possible shapes $\theta$, we obtain:
\[
|\YSSYT_k(\lambda/\mu, k\nu)| = 
\sum_{\theta \vdash nk} |\mathrm{LR}(\lambda/\mu, \theta)| \cdot |\YSSYT_k(\theta, k\nu)|,
\]
which yields the desired formula.
\end{proof}

\begin{theorem}\label{thm:schurExpansion}
Suppose $\lambda/\mu$ is a skew shape of size $nk$. 
Then
\begin{equation}
\rowMap_k(\schur_{\lambda/\mu}(\xvec)) = 
\sum_{\nu \vdash n} \; |\YSSYT_k(\lambda/\mu, k\nu)| \cdot \schur_{\nu}(\xvec). 
\end{equation}
In particular, for $\lambda\vdash kn$,
\begin{equation}
\rowMap_k(\schur_{\lambda}(\xvec)) = 
\sum_{\mu \vdash n} \; |\YSSYT_k(\lambda, k\mu) | \cdot \schur_{\mu}(\xvec).
\end{equation}
\end{theorem}
\begin{proof} 
This follows from Proposition~\ref{prop:kYamanouchiAsLRCoefficient}, Lemma~\ref{lem:skewYSSYTandLR}, and \eqref{eq:rowMapOfSkewSInS}.
\end{proof}

\begin{example}
Let us consider all $2$-Yamanouchi tableaux of shape $(4,4,2)$
and with partition contents $(4,4,2)$, $(4,2,2,2)$ and $(2,2,2,2,2)$.
These are the following tableaux:
\medskip 

\begin{center}
\ytableausetup{boxsize=0.95em}
\begin{tabular}{lccccccc}
\toprule
Content & Tableaux \\
\midrule 
442 & $\ytableaushort{1111,2222,33}$ \\ 
4222 & $\ytableaushort{1111,2234,34}$ \\
22222 & $\ytableaushort{1134,2255,34}$ & $\ytableaushort{1134,2245,35}$ &
$\ytableaushort{1124,2355,34}$ & $\ytableaushort{1124,2335,45}$ &
$\ytableaushort{1123,2445,35}$ & $\ytableaushort{1123,2345,45}$ & 
$\ytableaushort{1123,2344,55}$ \\
\bottomrule
\end{tabular}
\end{center}
\medskip 

The Schur expansion is 
\[
\rowMap_2(\schur_{442} ) = \schur_{221} + \schur_{2111} + 7 \schur_{11111},
\]
and the coefficients count the $2$-Yamanouchi tableaux.
\end{example}

\section{The expansion of \texorpdfstring{$\rowMap_k(\elementaryE_{\mu})$}{rowDiv of elementary}}\label{rowmap_kOne}

\subsection{Schur expansions}
We first relate the sum of the $\elementaryE$-coefficients to a Schur coefficient.
\begin{lemma}\label{lem:sumOfECoeffsIsSchurCoeff}
If $f$ is a homogeneous symmetric function of degree $n$ with
\[
f(\xvec) = \sum_{\alpha \vdash n} c_\alpha \elementaryE_\alpha(\xvec) = \sum_{\beta \vdash n} d_\beta \schur_\beta(\xvec),
\]
then $\sum_{\alpha} c_\alpha = d_{1^n}$.
\end{lemma}
\begin{proof}
Apply $\omega$ and take the Hall inner product with $\schur_{(n)}$. This gives
\[
 d_{1^n}=\sum_{\alpha\vdash n}c_\alpha
 \langle\completeH_\alpha,\schur_{(n)}\rangle
 =\sum_{\alpha\vdash n}c_\alpha,
\]
since $\langle\completeH_\alpha,\schur_{(n)}\rangle=K_{(n),\alpha}=1$.
\end{proof}

\begin{definition}\label{def:bcm}
A \defin{binary contingency matrix} with marginals $\lambda = (\lambda_1,\dotsc,\lambda_r)$ and 
$\mu = (\mu_1,\dotsc,\mu_c)$, is an $r \times c$ matrix where entries are in $\{0,1\}$
and the row sums are given by $\lambda$, and column sums given by $\mu$.
We denote the set of such matrices by \defin{$\BCM(\lambda,\mu)$}.
\end{definition}

\begin{lemma}[{See \cite[Thm. 7.4.4]{StanleyEC2}}]\label{lem:e-in-m}
The expansion of $\elementaryE_\lambda$ in the monomial basis is given by
\begin{equation*}
    \elementaryE_\lambda = \sum_{ \mu \dominatesBy \lambda'} |\BCM(\lambda,\mu)| \, \monomial_{\mu},
\end{equation*}
where $\BCM(\lambda,\mu)$ is as in Definition~\ref{def:bcm}.
\end{lemma}

Consider any total order on the partitions of $n$ that extends the dominance order.
For example, for $n=4$ we may take the total order
$1111$, $211$, $22$, $31$, $4$.
By Lemma~\ref{lem:e-in-m}, the matrix
\begin{equation}\label{eq:e-in-m-matrix}
   \left( \; |\BCM(\lambda',\mu)| \;  \right)_{\lambda,\mu}
\end{equation}
is lower-triangular. For example, $|\BCM(31',211)|=5$, and the full matrix is
\[
\left(
\begin{array}{ccccc}
 1 & 0 & 0 & 0 & 0 \\
 4 & 1 & 0 & 0 & 0 \\
 6 & 2 & 1 & 0 & 0 \\
 12 & \mathbf{5} & 2 & 1 & 0 \\
 24 & 12 & 6 & 4 & 1 \\
\end{array}
\right).
\]

\begin{theorem}\label{thm:emuSchurExpansion}
For $\mu \vdash kn$, we have
\begin{equation}\label{eq:emuInSchur}
\rowMap_k(\elementaryE_{ \mu }) = \sum_{\lambda \vdash n}
K_{a_\lambda',\,\mu}
\cdot \schur_\lambda,
\end{equation}
where $\ell=\ell(\lambda)$, $\rho=(\ell-1,\ldots,0)$, and
$\defin{a_\lambda }\coloneqq (k\lambda + (k-1)\rho) / (k-1)\rho$
and $K_{(\alpha/\beta)',\gamma}$ denotes the number
of semistandard Young tableaux of type $\gamma$ and skew shape
$(\alpha/\beta)'$.
In the special case $\mu = 1^{kn}$,
\begin{equation}\label{eq:e1knInSchur}
      \rowMap_k(\elementaryE_{1^{k n}}) = \sum_{\lambda \vdash n}
      f^{a_\lambda} \cdot \schur_\lambda,
\end{equation}
where $f^{\alpha/\beta}$ is the number of
standard Young tableaux of skew shape $\alpha/\beta$.
\end{theorem}
\begin{proof}
By the definition of the adjoint, for any $\lambda \vdash n$,
\[
[\schur_\lambda]\, \rowMap_k(\elementaryE_\mu)
= \langle \elementaryE_\mu,\, \rowAdj{k}(\schur_\lambda) \rangle.
\]
By Proposition~\ref{prop:rowMapAdjSchurExp},
$\rowAdj{k}(\schur_\lambda) = \schur_{a_\lambda}$
where $a_\lambda = (k\lambda + (k-1)\rho)/(k-1)\rho$.
Since $\langle \elementaryE_\mu, \schur_\gamma \rangle
= \langle \completeH_\mu, \omega(\schur_\gamma) \rangle
= \langle \completeH_\mu, \schur_{\gamma'} \rangle
= K_{\gamma',\mu}$ for any (skew) shape $\gamma$, we obtain
\[
[\schur_\lambda]\, \rowMap_k(\elementaryE_\mu)
= K_{a_\lambda',\,\mu},
\]
which gives \eqref{eq:emuInSchur}.
For $\mu=1^{kn}$, the Kostka coefficient is $K_{a_\lambda',1^{kn}}=f^{a_\lambda'}$. Transposing a standard tableau gives $f^{a_\lambda'}=f^{a_\lambda}$, and \eqref{eq:e1knInSchur} follows.
\end{proof}

\begin{example}
Let $k=2$ and $\mu = (2,1,1)$, so $n=2$ and $kn = 4$.
The partitions of $2$ are $\lambda = (2)$ and $\lambda = (1,1)$.

For $\lambda = (2)$, we have $\ell=1$ and $\rho = (0)$, so
$a_{(2)} = (4)$ and $a_{(2)}' = (1^4)$.
Since the entries in a column must be strictly increasing, there
is no tableau of shape $(1^4)$ and content $(2,1,1)$. Thus
$K_{(1^4),(2,1,1)}=0$.

For $\lambda=(1,1)$, we have $a_{(1,1)}=(3,2)/(1)$ and $a_{(1,1)}'=(2,2,1)/(1)$.
The semistandard Young tableaux of shape $(2,2,1)/(1)$ and content $(2,1,1)$ are
\[
 \ytableausetup{boxsize=1.2em}
 \ytableaushort{{\none}1,12,3}
 \qquad\text{and}\qquad
 \ytableaushort{{\none}1,13,2}
\]
Thus $\rowMap_2(\elementaryE_{2,1,1})=2\schur_{(1,1)}$.
\end{example}

\begin{remark}
For $\mu=1^{kn}$, the monomial expansion gives
\[
 \rowMap_k(\elementaryE_{1^{kn}})
 =\sum_{\nu\vdash n}|\BCM(1^{kn},k\nu)|\,\monomial_\nu.
\]
On the other hand, \eqref{eq:e1knInSchur} shows that the coefficient of $\monomial_\nu$ counts pairs in $\bigsqcup_{\lambda\vdash n}
\SYT(a_\lambda)\times\SSYT(\lambda,\nu)$.
A bijection from $\BCM(1^{kn},k\nu)$ to these pairs would prove \eqref{eq:e1knInSchur} combinatorially. A variant of RSK might provide such a bijection. The usual RSK correspondence produces tableaux of straight shape, whereas $a_\lambda$ is generally skew.

A proof using crystal operators also faces a difficulty. For $k>1$, the usual crystal operators need not preserve words whose content is divisible by $k$. They therefore do not directly prove Schur positivity of $\rowMap_k(\elementaryE_\mu)$.
\end{remark}

\begin{definition}
Let $k\ge1$ and $\mu\vdash kn$. The \defin{$(k,\mu)$-Euler number} $E_{k,\mu}$ is the sum of the coefficients in the
$\elementaryE$-expansion of $\rowMap_k(\elementaryE_\mu)$. That is, if
\[
 \rowMap_k(\elementaryE_\mu)
 =\sum_{\nu\vdash n}c_{\mu,\nu}\elementaryE_\nu,
\]
then
\[
 E_{k,\mu}\coloneqq\sum_{\nu\vdash n}c_{\mu,\nu}.
\]
\end{definition}

\begin{proposition}
Let $k\geq1$ and $\mu\vdash nk$, with $\ell=\ell(\mu)$ and $\rho=(n-1,\ldots,0)$.
Then the following sets have cardinality given by $E_{k,\mu}$:
\begin{enumerate}
\item[(a)] 
$\SSYT( ((k^n + (k-1)\rho) / (k-1)\rho)',\mu)$
\item[(b)] 
$k$-Yamanouchi skew tableaux with disjoint columns of sizes $\mu_1,\mu_2,\dotsc$
and content $1^{k}\,2^{k} \, \dotsc \, n^{k}$
\item[(c)] 
Multiset permutations $w=w_1\cdots w_{nk}$ of $1^{\mu_1}\, 2^{\mu_2}\, \dotsc \, \ell^{\mu_\ell}$
with $w_i>w_{i+1}$ when $k\nmid i$ and $w_i\le w_{i+1}$ when $k\mid i$, for $1\le i<nk$.
\end{enumerate}
\end{proposition}
\begin{proof}
The first and second items follow from Theorem~\ref{thm:schurExpansion} and 
\eqref{eq:emuInSchur}  combined with Lemma~\ref{lem:sumOfECoeffsIsSchurCoeff}.
For part~(b), place column $j$ in rows
$1+\sum_{h>j}\mu_h,\ldots,\sum_{h\ge j}\mu_h$.
The columns run from southwest to northeast and share no rows, so their
skew Schur function is $\elementaryE_\mu$. The third statement follows from a bijection between (a) and (c).

Given $T\in \SSYT(((k^n + (k-1)\rho) / (k-1)\rho)',\mu)$, 
we read $T$ column by column, from bottom to top to obtain the word
\[
w(T)=w_1w_2\cdots w_{nk}.
\]
Since each column is increasing, it follows that for every $1\le j\le n$, the subword $w_{(j-1)k+1},\dots,w_{jk}$ is descending. 
For $1\le j<n$, the last cell of column $j$ is immediately
to the left of the first cell of column $j+1$. Semistandardness implies that
$w_{jk}\le w_{jk+1}$. Thus $w(T)$ is a permutation 
of $1^{\mu_1}2^{\mu_2}\cdots \ell^{\mu_\ell}$ in which all multiples 
of $k$ are weak ascents.
\medskip 

Conversely, given such a word $w$, split it into $n$ consecutive blocks of length $k$, and place the $j$th block into column $j$ of the
shape $((k^n+(k-1)\rho)/(k-1)\rho)^\prime$, filling each column from bottom to top. The condition
$w_{jk}\le w_{jk+1}$ guarantees that each row is weakly increasing.
\end{proof}
The number $E_{k,\mu}$ is invariant under permutation of the entries of $\mu$. 

\begin{corollary}
For $k=2$, the number $E_{2,1^{2n}}$ counts down-up permutations of size $2n$. These are Euler zigzag numbers; see \oeis{A000111}.    
\end{corollary}

\subsection{Connection to work of Amdeberhan, Shareshian, and Stanley}

We relate these results to work of Amdeberhan, Shareshian, and Stanley on alternating permutations and Euler numbers; see~\cite{AmdeberhanShareshianStanley2026}.

\begin{lemma}\label{lem:omegaOfGenFunc}
Let $f$ be any formal power series with $f(0)=1$. Then
\begin{equation}\label{eq:FPS}
\omega\Bigl(\prod_i f(x_i)\Bigr)=\prod_i f(-x_i)^{-1}.
\end{equation}
\end{lemma}
\begin{proof}
Write $\log f(z) = \sum_{r\ge1} a_r z^r$. Then
\[
\prod_i f(x_i)=\exp(\sum_{r\ge1} a_r \powersum_r(\xvec)).
\]
Since $\omega(\powersum_r)=(-1)^{r-1}\powersum_r$,
\[
\omega\Bigl(\prod_i f(x_i) \Bigr) = \exp(\sum_{r\ge1} a_r (-1)^{r-1}\powersum_r(\xvec))=\prod_i f(-x_i)^{-1}.
\]
\end{proof}

\begin{definition}
Define $\defin{A_{n,k}(\xvec)}$ by
\[
\sum_{n\ge 0}\frac{A_{n,k}(\xvec)t^n}{(nk)!}=\prod_{i\ge 1}\Biggl(\sum_{m\ge 0}\frac{(-1)^m x_i^{m}t^{m}}{(km)!}\Biggr)^{-1}.
\]
\end{definition}
When $k=2$, the series inside the product is $\cos(\sqrt{x_it})$. Thus, if $A_n$ denotes the symmetric function in \cite{AmdeberhanShareshianStanley2026}, its generating function with seed $\sec(\sqrt{t})$ gives
\[
 A_{n,2}=(2n)!A_n.
\]

\begin{definition}[see~\cite{AMDEBERHAN_ONO_SINGH_2025}]
For $\lambda = \langle 1^{m_1}, \dots, n^{m_n} \rangle \vdash n$, define
\[
\defin{\phi(\lambda)} \coloneqq(2n)! \cdot \prod_{k=1}^n \frac{1}{m_k!} \left( \frac{4^k(4^k - 1) B_{2k}}{(2k)(2k)!} \right)^{m_k},
\]
where $B_{2k}$ is the $2k$th Bernoulli number.
\end{definition}
By the definition of $A_n$ in~\cite{AmdeberhanShareshianStanley2026} and the relation above,
\begin{equation}\label{eq:ass-powersum}
 A_{n,2}(\xvec) = \sum_{\lambda\vdash n} |\phi(\lambda)| \powersum_\lambda(x)
\end{equation}
where $\sum_{\lambda\vdash n} |\phi(\lambda)|$ is the number of alternating permutations $w\in \symS_{2n}$ 
whose ascent set is exactly $\{2,4,\dotsc,2n-2\}$; see \cite[Theorem~3.2]{AmdeberhanShareshianStanley2026}.

\begin{theorem}[{\cite[Theorem~5.1]{AmdeberhanShareshianStanley2026}}]\label{thm:Stanleyhpositive}
The symmetric function $A_{n,2}(\xvec)$ is $\completeH$-positive.
\end{theorem}

\begin{theorem}
For $k\geq 1$, we have
\begin{equation}\label{eq:sameAsStanley}
   \omega(\rowMap_k(\elementaryE_{1^{nk}} (\xvec))) =A_{n,k}(\xvec).
\end{equation}
\end{theorem}
\begin{proof}
The multinomial theorem gives
\[
\elementaryE_1^{nk}(\xvec)
=\sum_{\alpha_1+\alpha_2+\dotsb=nk}\frac{(nk)!}{\alpha_1!\alpha_2!\dotsm} x_1^{\alpha_1}x_2^{\alpha_2}\dotsm.
\]
Applying $\rowMap_k$ retains precisely those terms for which each $\alpha_i$ is divisible by $k$, say $\alpha_i=k\beta_i$, and then divides the exponents by $k$. Hence
\[
\rowMap_k(\elementaryE_1^{nk}(\xvec))
=\sum_{\beta_1+\beta_2+\cdots=n}\frac{(nk)!}{(k\beta_1)!(k\beta_2)!\cdots} x_1^{\beta_1}x_2^{\beta_2}\cdots.
\]
We consider
\[
\sum_{n\ge 0}\frac{\rowMap_k(\elementaryE_1^{nk}(\xvec))t^n}{(nk)!}
=\sum_{\beta_1,\beta_2,\ldots\ge 0}\prod_{i\ge 1}\frac{(x_it)^{\beta_i}}{(k\beta_i)!}
=\prod_{i\ge 1}\Biggl(\sum_{m\ge 0}\frac{(x_it)^m}{(km)!}\Biggr).
\]
Let $\defin{\mittagLefflerE_k(z)} \coloneqq \sum_{m\ge 0} \frac{z^m}{(mk)!}$ be the \defin{Mittag-Leffler function}. Then
\[
\sum_{n\ge 0}\frac{\rowMap_k(\elementaryE_1^{nk}(\xvec))t^n}{(nk)!}=\prod_{i\ge 1}\mittagLefflerE_k(x_it).
\]
We then apply $\omega$ and define
\begin{equation}\label{eq:genfun_row(e_1^{nk})}
\defin{F_k(x,t)}  \coloneqq \omega\Bigl(\prod_{i\ge 1}\mittagLefflerE_k(x_it)\Bigr).
\end{equation}
By Lemma~\ref{lem:omegaOfGenFunc},
with $f(z)=\mittagLefflerE_k(zt)$ in \eqref{eq:genfun_row(e_1^{nk})}, we obtain
\[
F_k(x,t)
=\omega\Bigl(\prod_i \mittagLefflerE_k(x_i t)\Bigr)
=\prod_i \mittagLefflerE_k(-x_i t)^{-1}
=\prod_i\Bigl(\sum_{m\ge 0}\frac{(-1)^m x_i^m t^m}{(km)!}\Bigr)^{-1}.
\]
This proves \eqref{eq:sameAsStanley}.
\end{proof}
Together with the elementary-positivity theorem in~\cite{AlexanderssonDai2026Epositive}, this identity implies that $A_{n,k}$ is $\completeH$-positive for every $k\ge1$.

\subsection{Power-sum expansions}\label{sec:powerSumExpansion}

It is natural to ask whether \eqref{eq:ass-powersum} generalizes to general $\mu\vdash nk$.
We give a combinatorial interpretation of the coefficients in the expansion of
$\omega(\rowMap_k(\elementaryE_{\mu}))$.

\begin{lemma}\label{lem:omega-rowmap-gf}
Fix $k\geq 1$ and let
\[
\defin{\mathcal F_k(\xvec,\yvec)} \coloneqq \sum_\mu \omega\bigl(\rowMap_k(\elementaryE_\mu(\xvec))\bigr)\monomial_\mu(\yvec),
\]
where the sum is over all partitions $\mu$. Then
\[
\mathcal F_k(\xvec,\yvec) = \prod_i   B_k(x_i;\yvec)^{-1} 
= 
\exp\left(\sum_{r\ge 1}\Phi_{k,r}(\yvec)\frac{\powersum_r(\xvec)}{r}\right),
\]
where $\defin{B_k(t;\yvec)}\coloneqq \sum_{m\ge 0}(-1)^m \elementaryE_{km}(\yvec)t^m$
and $\defin{\Phi_{k,r}(\yvec)}\coloneqq -r[t^r]\log B_k(t;\yvec)$.
\end{lemma}
\begin{proof}
Introduce
\[
\defin{A_k(t;\yvec)} \coloneqq \sum_{m\ge 0} \elementaryE_{km}(\yvec)t^m
\qquad\text{and}\qquad
\defin{B_k(t;\yvec)} \coloneqq A_k(-t;\yvec)=\sum_{m\ge 0}(-1)^m \elementaryE_{km}(\yvec)t^m.
\]

The second Cauchy identity (see \cite{StanleyEC2}) states that
\[
\sum_\mu \elementaryE_\mu(\xvec)\monomial_\mu(\yvec)
=
\prod_{i,j}(1+x_i y_j) =
\prod_i\left(\sum_{r\ge 0} \elementaryE_r(\yvec)x_i^r\right).
\]
Expand the product as a series in the $x_i$ and apply $\rowMap_k$ to each monomial. This gives
\[
\sum_\mu \rowMap_k(\elementaryE_\mu(\xvec))\monomial_\mu(\yvec)
=\prod_i\left( \sum_{m\ge 0} \elementaryE_{km}(\yvec)x_i^m  \right)
=
\prod_i A_k(x_i;\yvec).
\]
Applying $\omega$ on the $\xvec$-alphabet
together with Lemma~\ref{lem:omegaOfGenFunc} gives that 
\begin{equation}\label{eq:pexpStep}
\mathcal F_k(\xvec,\yvec)
=
\omega\left(\prod_i A_k(x_i;\yvec)\right)=
\prod_i B_k(x_i;\yvec)^{-1}.
\end{equation}
Note that $\Phi_{k,r}(\yvec)$ satisfies
\[
-\log B_k(t;\yvec)=\sum_{r\ge 1}\Phi_{k,r}(\yvec)\frac{t^r}{r}.
\]
Take logarithms of each factor in \eqref{eq:pexpStep} and then exponentiate. This gives
\[
\mathcal F_k(\xvec,\yvec)=\prod_i B_k(x_i;\yvec)^{-1}
=
\exp\left(\sum_i\sum_{r\ge 1}\Phi_{k,r}(\yvec)\frac{x_i^r}{r}\right)
=
\exp\left(\sum_{r\ge 1}\Phi_{k,r}(\yvec)\frac{\powersum_r(\xvec)}{r}\right).
\]
\end{proof}
\begin{corollary}\label{cor:p-explicit}
For every partition $\mu\vdash nk$,
if the $a_{\mu,\nu}^{(k)}$ are defined via the relation
\[
\omega\bigl(\rowMap_k(\elementaryE_\mu(\xvec))\bigr)
=
\sum_{\nu\vdash n}\frac{\defin{a_{\mu,\nu}^{(k)}}}{z_\nu}\powersum_\nu(\xvec),\quad \text{then}\quad a_{\mu,\nu}^{(k)}
=
[\monomial_\mu(\yvec)]
\prod_{j=1}^{\ell(\nu)} \Phi_{k,\nu_j}(\yvec).
\]
\end{corollary}
\begin{proof}
By Lemma~\ref{lem:omega-rowmap-gf},
\[
\mathcal F_k(\xvec,\yvec)
=
\exp\left(\sum_{r\ge 1}\Phi_{k,r}(\yvec)\frac{\powersum_r(\xvec)}{r}\right)
=
\sum_\nu \frac{\powersum_\nu}{z_\nu}\prod_{j=1}^{\ell(\nu)} \Phi_{k,\nu_j}(\yvec).
\]
On the other hand,
\[
\mathcal F_k(\xvec,\yvec)
=
\sum_\mu \omega\bigl(\rowMap_k(\elementaryE_\mu(\xvec))\bigr)\monomial_\mu(\yvec).
\]
Comparing coefficients of $\monomial_\mu(\yvec)$ proves the claim.
\end{proof}

We introduce a symmetric function $f_{k,r}$ for the $\powersum$-expansion of $\omega(\rowMap_k(\elementaryE_{\mu}))$.
\begin{definition}
    For integers $k,r\geq 1$, define the symmetric function
\begin{equation}\label{eq:fkr}
  \defin{f_{k,r}}
  \;\coloneqq\;
  \sum_{j=1}^{r} (-1)^{r+j}\,\frac{r}{j}
  \sum_{\alpha\,\models_j\, r} \elementaryE_{k\alpha_1}\cdots \elementaryE_{k\alpha_j},
\end{equation}
where the inner sum is over compositions
$\alpha=(\alpha_1,\dots,\alpha_j)$ of~$r$ into~$j$ positive parts
and $\elementaryE_n$ denotes the $n$-th elementary symmetric function.
\end{definition}

\begin{definition}
Let $k,r\ge1$. A \defin{$(k,r)$-cyclic word} is a word
$w=w_1w_2\cdots w_{kr}$ of positive integers such that
\[
 w_i<w_{i+1}\quad\text{if }k\nmid i,
 \qquad
 w_i\ge w_{i+1}\quad\text{if }k\mid i,
\]
where $1\le i\le kr$ and $w_{kr+1}=w_1$. Words are counted as sequences; rotations are not identified.
\end{definition}
For example, a $(2,3)$-cyclic word satisfies
\[
 w_1<w_2\ge w_3<w_4\ge w_5<w_6\ge w_1.
\]
The \defin{weight} of a word $w=w_1\cdots w_{kr}$ is
\[
 \mathbf{x}^w\coloneqq x_{w_1}x_{w_2}\cdots x_{w_{kr}}.
\]
We will show that $f_{k,r}$ is the sum of the weights of all
$(k,r)$-cyclic words.

\begin{lemma}\label{lem:double-count}
For $1\le j\le r$, we have
\[
 \sum_{\substack{S\subseteq[r]\\|S|=j}}
 \elementaryE_{k\alpha(S)}
 =
 \frac{r}{j}\sum_{\alpha\,\models_j\,r}\elementaryE_{k\alpha},
\]
For $S=\{s_1<\cdots<s_j\}$, the cyclic gap composition is
\[
 \alpha(S)=(s_2-s_1,\ldots,s_j-s_{j-1},r+s_1-s_j),
\]
with $\alpha(\{s\})=(r)$. The product $\elementaryE_{k\alpha(S)}$
does not depend on the starting element chosen in $S$.
\end{lemma}

\begin{proof}
Choose a starting element $b\in S$ and read the distances around
the cycle. This gives a composition $\alpha$ of $r$ into $j$ parts.
A different choice of $b$ cyclically permutes the parts of $\alpha$,
so it does not change the product
$\elementaryE_{k\alpha_1}\cdots\elementaryE_{k\alpha_j}$.

Each set $S$ has $j$ choices of $b$. Conversely, a composition
$\alpha$ of $r$ into $j$ parts and a starting point $b\in[r]$
determine $S$. Counting the pairs $(S,b)$ gives
\[
 j\sum_{\substack{S\subseteq[r]\\|S|=j}}
 \elementaryE_{k\alpha(S)}
 =
 r\sum_{\alpha\,\models_j\,r}\elementaryE_{k\alpha}. \qedhere
\]
\end{proof}

\begin{theorem}\label{thm:mpos}
For $k,r\ge1$, we have
 $f_{k,r}
 =\sum_{w}\mathbf{x}^w$, 
where the sum is over all $(k,r)$-cyclic words $w$. In particular,
$f_{k,r}$ is monomial-positive.
\end{theorem}

\begin{proof}
Let $\mathcal W$ be the set of words $w=w_1\cdots w_{kr}$ such that $w_i<w_{i+1}$ whenever $1\le i<kr$ and $k\nmid i$. There is no restriction at the other positions. The weight generating function of $\mathcal W$ is $\elementaryE_k^r$.

For $i\in[r]$, let $A_i$ be the set of words in $\mathcal W$ with $w_{ik}<w_{ik+1}$, where indices are read modulo $kr$. All
complements below are taken in $\mathcal W$. The $(k,r)$-cyclic words are precisely the
elements of $W\setminus\bigcup_{i=1}^r A_i$.

The inclusion--exclusion principle gives
\begin{equation}\label{eq:ie}
 \sum_{w\in\overline{A_1}\cap\cdots\cap\overline{A_r}}
 \mathbf{x}^w
 =
 \sum_{S\subseteq[r]}(-1)^{|S|}G_S,
\end{equation}
where $G_S$ is the weight generating function for words in
$\mathcal W$ with $w_{ik}<w_{ik+1}$ for every $i\in S$.

Suppose $|S|=r-j$. Each ascent at a boundary in $S$ joins two
adjacent blocks. The resulting $j$ increasing blocks have lengths
$k\alpha_1,\ldots,k\alpha_j$, where
$\alpha=\alpha([r]\setminus S)$. Hence
\[
 G_S=\elementaryE_{k\alpha_1}\cdots\elementaryE_{k\alpha_j}.
\]
If $S=[r]$, the required ascents go all the way around the cycle, so $G_S=0$.

Grouping the terms in \eqref{eq:ie} by $j=r-|S|$ and applying
Lemma~\ref{lem:double-count}, we obtain
\[
 \begin{aligned}
 \sum_{w\text{ a }(k,r)\text{-cyclic word}}\mathbf{x}^w
 =\sum_{j=1}^r(-1)^{r-j}
   \sum_{\substack{S\subseteq[r]\\|S|=r-j}}
   \elementaryE_{k\alpha([r]\setminus S)}
 =\sum_{j=1}^r(-1)^{r+j}\frac{r}{j}
   \sum_{\alpha\,\models_j\,r}\elementaryE_{k\alpha}
 =f_{k,r}.
 \end{aligned}
\]
Thus $f_{k,r}$ is the weight generating function for
$(k,r)$-cyclic words. In particular, it is monomial-positive.
\end{proof}

\begin{corollary}\label{coro:p-positive}

For integers $k,r\geq 1$,
   \[ 
    \Phi_{k,r}(\mathbf y)=f_{k,r}(\mathbf y).
    \]
   Therefore, $\omega(\rowMap_k(\elementaryE_{\mu}))$ is $\powersum$-positive.
\end{corollary}
\begin{proof}
All symmetric functions in this proof use the alphabet $\mathbf y$. By definition,
\[
 \Phi_{k,r}(\mathbf y)
 =-r[t^r]\log\left(
 \sum_{m\ge0}(-1)^m\elementaryE_{km}(\mathbf y)t^m
 \right).
\]
Using
\[
 -\log(1+u)=\sum_{j\ge1}\frac{(-1)^j}{j}u^j
\]
and taking the coefficient of $t^r$, we obtain
\[
 \Phi_{k,r}(\mathbf y)
 =r\sum_{j=1}^r\frac{(-1)^j}{j}
 \sum_{\substack{m_1+\cdots+m_j=r\\m_1,\ldots,m_j\ge1}}
 (-1)^{m_1+\cdots+m_j}
 \elementaryE_{km_1}\cdots\elementaryE_{km_j}.
\]
Since $m_1+\cdots+m_j=r$, this becomes
\[
 \Phi_{k,r}(\mathbf y)
 =\sum_{j=1}^r(-1)^{r+j}\frac{r}{j}
 \sum_{\substack{m_1+\cdots+m_j=r\\m_1,\ldots,m_j\ge1}}
 \elementaryE_{km_1}\cdots\elementaryE_{km_j}
 =f_{k,r}(\mathbf y).
\]

By Theorem~\ref{thm:mpos}, each $f_{k,r}$ is monomial-positive.
Their products are also monomial-positive. Hence
Corollary~\ref{cor:p-explicit} gives nonnegative
$\powersum$-coefficients.
\end{proof}
\begin{remark}
For $k\ge2$, the function $f_{k,r}$ is a cylindric ribbon
Schur function in the sense of
\cite{McNamara2006}.
Here is the periodic shape in English notation.
For every $j\in\setZ$, let a column of $k$ cells in column
$-j$, occupying rows $(k-1)j+1,\ldots,(k-1)j+k$.
Identify cells under translation by $(-r,(k-1)r)$, where the coordinates are (column, row). The quotient has $kr$ cells.
Adjacent columns meet in one row, so the periodic lift is a connected ribbon with no $2\times2$ square.

Read column $-j$ from top to bottom as $w_{jk+1},\ldots,w_{(j+1)k}$, with indices taken modulo $kr$. The column and row conditions give
\[
w_{jk+1}<\cdots<w_{(j+1)k},
\qquad
w_{(j+1)k}\ge w_{(j+1)k+1}.
\]
The second inequality comes from the common row of columns $-j$ and $-j-1$. Thus semistandard fillings of the cylindric ribbon are in weight-preserving bijection with $(k,r)$-cyclic words.

For $k=1$, the cyclic inequalities force $w_1=\cdots=w_r$, giving $f_{1,r}=\powersum_r$.
\end{remark}

\subsection{The \texorpdfstring{$\elementaryE$}{e}-expansion}

Recall $a_\lambda=(k\lambda+(k-1)\rho)/(k-1)\rho$. 
Using the standard involution $\omega$ on symmetric functions, defined by
$\omega(\completeH_\alpha)=\elementaryE_\alpha$ and $\omega(\schur_\lambda)=\schur_{\lambda'}$, we have
\[
[\elementaryE_\nu]\rowMap_k(\elementaryE_\mu) = [\completeH_\nu]\omega(\rowMap_k(\elementaryE_\mu)).
\]
Applying $\omega$ to \eqref{eq:emuInSchur} yields
\begin{equation}\label{eq:hCoeff}
\omega(\rowMap_k(\elementaryE_\mu))
= \sum_{\lambda \vdash n}
K_{a_\lambda',\,\mu}\; \schur_{\lambda'} 
\quad 
\implies 
\quad 
[\completeH_\nu]\omega(\rowMap_k(\elementaryE_\mu))
=
\sum_{\lambda \vdash n}
K_{a_\lambda',\,\mu}\,
[\completeH_\nu] \schur_{\lambda'}.
\end{equation}
Using the inverse Kostka expansion
\begin{equation}
    \schur_{\lambda'} = \sum_{\nu} K^{-1}_{\nu,\lambda'} \completeH_\nu,
\end{equation}
we obtain
\begin{equation}\label{eq:keyFormula}
[\elementaryE_\nu]\rowMap_k(\elementaryE_\mu)
=
\sum_{\lambda \vdash n}
K_{a_\lambda',\,\mu}
\cdot K^{-1}_{\nu,\lambda'}.
\end{equation}

Recall (see \cite{EgeciogluRemmel1990}) 
that the inverse Kostka number
$K^{-1}_{\nu,\lambda'}$ can be interpreted as a signed enumeration of
\emph{special rim hook tableaux} ($\SRHT$) of shape $\lambda'$ and type $\nu$:
\[
K^{-1}_{\nu,\lambda'}
=
\sum_{S \in \SRHT(\lambda',\nu)} \sgn(S),
\]
where the sign of $S$ is given by
\[
\defin{\sgn(S)} \coloneqq (-1)^{\sum_i (\text{number of rows of the $i$-th rim hook} - 1)}.
\]

On the other hand, the coefficient
$K_{a_\lambda',\,\mu}$ counts semistandard Young tableaux
of skew shape $a_\lambda'$ and content $\mu$.
Therefore \eqref{eq:keyFormula} becomes
\begin{equation}\label{eq:signedCount}
[\elementaryE_\nu]\rowMap_k(\elementaryE_\mu)
=
\sum_{\lambda \vdash n}
\sum_{T \in \SSYT(a_\lambda',\mu)}
\sum_{S \in \SRHT(\lambda',\nu)}
\sgn(S).
\end{equation}

Thus the $\elementaryE$-coefficient is expressed as a signed count of triples
\[
(\lambda, T, S),
\]
where $T$ is a semistandard Young tableau of skew shape
$a_\lambda'$ and content $\mu$, and
$S$ is a special rim hook tableau of shape $\lambda'$ and type $\nu$.

\begin{example}
Let $\mu = 22211$, $k=2$. Then 
in order to compute $\rowMap_{2}(\elementaryE_\mu)$ 
we need to consider the following skew shapes:
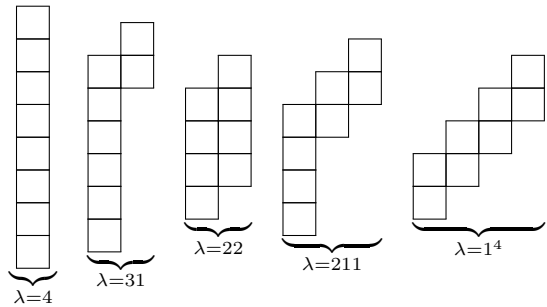
\[
\underbrace{\ydiagram{1, 1, 1, 1, 1,1,1,1}}_{\lambda=4}
\quad
\underbrace{\ydiagram{1+1, 2, 1, 1, 1,1,1}}_{\lambda=31}
\quad 
\underbrace{\ydiagram{1+1,2, 2,2,1}}_{\lambda=22}
\quad 
\underbrace{\ydiagram{2+1,1+2, 2,1,1,1}}_{\lambda=211}
\quad 
\underbrace{\ydiagram{3+1, 2+2, 1+2, 2, 1}}_{\lambda=1^4}
\]
The coefficient of $\elementaryE_{31}$ is then the sum over all
semistandard Young tableaux of the shapes with weight $22211$,
each combined with a special rim hook tableau of shape $\lambda'$ with content $31$.
Only the shapes indexed by $\lambda\in\{22,211,1^4\}$ admit such fillings. Also, only $\lambda\in\{4,31,22,211\}$ give special rim hook tableaux of type $\nu=31$. Thus only two cases contribute.

\medskip 
\noindent
\textbf{Case $\lambda=22$ (negative):} There are two semistandard Young tableaux and one special rim hook tableau. 
Note that the sign is negative.
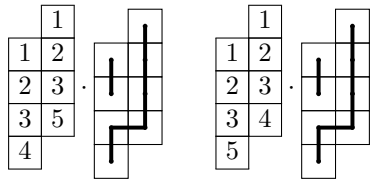
\[
\ytableaushort{{\none}1,12,23,35,4} 
\cdot 
\rimHookTableau[scale=0.45,Dot]{21,11}
\qquad 
\ytableaushort{{\none}1,12,23,34,5} \cdot 
\rimHookTableau[scale=0.45,Dot]{21,11}
\]

\medskip 
\noindent
\textbf{Case $\lambda=211$ (positive):}
There is one special rim hook tableau and 16 semistandard fillings:
\[
\rimHookTableau[scale=0.45,Dot]{111,2}
\]
\begin{align*}
&\begin{ytableau}
\none & \none & 3\\ \none & 1 & 5\\ 1 & 2\\ 2\\ 3\\ 4\\
\end{ytableau}\quad
\begin{ytableau}
\none & \none & 3\\ \none & 1 & 4\\ 1 & 2\\ 2\\ 3\\ 5\\
\end{ytableau}\quad
\begin{ytableau}
\none & \none & 2\\ \none & 1 & 5\\ 1 & 3\\ 2\\ 3\\ 4\\
\end{ytableau}\quad
\begin{ytableau}
\none & \none & 2\\ \none & 1 & 4\\ 1 & 3\\ 2\\ 3\\ 5\\
\end{ytableau}\quad
\begin{ytableau}
\none & \none & 2\\ \none & 1 & 3\\ 1 & 5\\ 2\\ 3\\ 4\\
\end{ytableau}\quad
\begin{ytableau}
\none & \none & 2\\ \none & 1 & 3\\ 1 & 4\\ 2\\ 3\\ 5\\
\end{ytableau}\quad
\begin{ytableau}
\none & \none & 2\\ \none & 1 & 3\\ 1 & 3\\ 2\\ 4\\ 5\\
\end{ytableau}\quad
\begin{ytableau}
\none & \none & 2\\ \none & 1 & 3\\ 1 & 2\\ 3\\ 4\\ 5\\
\end{ytableau}\\[1em]
&\begin{ytableau}
\none & \none & 1\\ \none & 2 & 5\\ 1 & 3\\ 2\\ 3\\ 4\\
\end{ytableau}\quad
\begin{ytableau}
\none & \none & 1\\ \none & 2 & 4\\ 1 & 3\\ 2\\ 3\\ 5\\
\end{ytableau}\quad
\begin{ytableau}
\none & \none & 1\\ \none & 2 & 3\\ 1 & 5\\ 2\\ 3\\ 4\\
\end{ytableau}\quad
\begin{ytableau}
\none & \none & 1\\ \none & 2 & 3\\ 1 & 4\\ 2\\ 3\\ 5\\
\end{ytableau}\quad
\begin{ytableau}
\none & \none & 1\\ \none & 2 & 3\\ 1 & 3\\ 2\\ 4\\ 5\\
\end{ytableau}\quad
\begin{ytableau}
\none & \none & 1\\ \none & 2 & 2\\ 1 & 3\\ 3\\ 4\\ 5\\
\end{ytableau}\quad
\begin{ytableau}
\none & \none & 1\\ \none & 1 & 3\\ 2 & 2\\ 3\\ 4\\ 5\\
\end{ytableau}\quad
\begin{ytableau}
\none & \none & 1\\ \none & 1 & 2\\ 2 & 3\\ 3\\ 4\\ 5\\
\end{ytableau}
\end{align*}

A direct computation gives
\[
 \rowMap_{2}(\elementaryE_\mu) = 136\elementaryE_4 + 2 \elementaryE_{22} + 14 \elementaryE_{31}
\]
and the coefficient of $\elementaryE_{31}$ agrees with the computation above.
\end{example}
The first version of this paper conjectured that
$\rowMap_k(\elementaryE_\mu)$ is $\elementaryE$-positive for every
$\mu\vdash kn$. This is now a theorem~\cite{AlexanderssonDai2026Epositive}.
We retain the coefficient identities below, which give a dual formulation
and unweighted interpretations for certain sums and special families.

\begin{proposition}\label{prop:e-m-duality}
Fix $k,n\ge1$ and let $\mathcal T_k=\omega\rowAdj{k}\omega$. The functions $\rowMap_k(\elementaryE_\mu)$ are $\elementaryE$-positive for all $\mu\vdash kn$ if and only if $\mathcal T_k(\monomial_\nu)$ is monomial-positive for every $\nu\vdash n$. More precisely,
\[
[\elementaryE_\nu]\rowMap_k(\elementaryE_\mu)
=[\monomial_\mu]\mathcal T_k(\monomial_\nu).
\]
\end{proposition}
\begin{proof}
By Proposition~\ref{prop:adjMap}, $\mathcal T_k$ is a ring homomorphism
and $\mathcal T_k(\elementaryE_r)=\elementaryE_{kr}$.
Apply $\mathcal T_k$ to the $\yvec$-alphabet in the second Cauchy identity.
This gives
\[
 \sum_\nu\elementaryE_\nu(\xvec)
       \mathcal T_k(\monomial_\nu(\yvec))
 =\prod_i\left(\sum_{r\ge0}\elementaryE_{kr}(\yvec)x_i^r\right).
\]
As in the proof of Lemma~\ref{lem:omega-rowmap-gf}, the product on
the right is also obtained by expanding the Cauchy product in monomials
and applying $\rowMap_k$ to the $\xvec$-alphabet. Hence
\[
 \sum_\nu\elementaryE_\nu(\xvec)
       \mathcal T_k(\monomial_\nu(\yvec))
 =\sum_\mu\rowMap_k(\elementaryE_\mu(\xvec))
       \monomial_\mu(\yvec).
\]
Comparing coefficients of
$\elementaryE_\nu(\xvec)\monomial_\mu(\yvec)$ proves the identity
and the equivalence.
\end{proof}

\subsection{Stanley's linear functional \texorpdfstring{$\stanleyMap$}{Theta} on quasisymmetric functions}
We use Stanley's linear functional on quasisymmetric functions~\cite{Stanley1995} to study sums of $\elementaryE$-coefficients.

\begin{definition}
For $k\geq 1$ and $0 \leq i < n$, we let \defin{$\SH(i)$} denote the ribbon skew shape  
where the number of boxes in the columns are $(k(i+1), k, k,\dotsc,k)$ and the total number 
of boxes is $kn$. In other words, $\SH(i)$ is the conjugate of 
the skew shape $(k\lambda+(k-1)\rho)/(k-1)\rho$ where $\lambda=(i+1,1^{n-i-1})$.
\emph{Note that we suppress the dependence on the parameters $n$ and $k$,
as these are clear from the context.}
\end{definition}

Let $F_{n,S}(\xvec)$ denote the fundamental quasisymmetric function of degree $n$ indexed by a subset
$S\subseteq\{1,2,\dots,n-1\}$. 
Recall the $\mathbb Q$-linear map $\stanleyMap:\QSym\to \mathbb{Q}[t]$, with $\stanleyMap(1)=1$, defined in positive degrees by
\begin{equation}
    \defin{\stanleyMap\left(F_{n+1,S}(\xvec)\right)} \coloneqq
\begin{cases}
t(t-1)^i, & \text{if } S=\{i+1,i+2,\dots,n\}\text{ for some }0\le i\le n,\\
0, & \text{otherwise}.
\end{cases}
\end{equation}
The map $\stanleyMap$ considered as a map on symmetric functions
is a ring homomorphism with the property that
\begin{equation}\label{eq:phiIne}
\stanleyMap(\elementaryE_\mu)=t^{\ell(\mu)}.
\end{equation}
Moreover, if $\lambda\vdash (n+1)$, then
\begin{equation}
    \stanleyMap\left(\schur_\lambda(\xvec)\right)=
\begin{cases}
t(t-1)^i & \text{if }\lambda=(i+1,1^{n-i})\text{ for some }0\le i\le n,\\
0 & \text{otherwise}.
\end{cases}
\end{equation}
By \eqref{eq:phiIne}, for any symmetric function $F=\sum_\mu c_\mu \elementaryE_\mu$,
\begin{equation}
    \stanleyMap(F)=\sum_\mu c_\mu t^{\ell(\mu)}
=\sum_{r\ge 0}\left(\sum_{\ell(\mu)=r} c_\mu\right)t^r,
\end{equation}
so the coefficient $[t^r]\stanleyMap(F)$ equals the sum of the $\elementaryE$-coefficients of $F$ over 
all partitions of length $r$. 
This map has an important role for chromatic symmetric functions, see \cite{Stanley1995}.

For $\mu\vdash kn$, we know the Schur expansion of $\rowMap_k(\elementaryE_{ \mu })$ from \eqref{eq:emuInSchur},
\begin{equation}
\rowMap_k(\elementaryE_{\mu }) = \sum_{\lambda \vdash n}
K_{a_{\lambda}',\mu}
\cdot \schur_\lambda,
\end{equation}
where $a_{\lambda}=(k\lambda+(k-1)\rho)/(k-1)\rho$.

Applying $\stanleyMap$ to this Schur expansion, we have
\begin{equation}\label{eq:phiInrowMap}
\stanleyMap\big(\rowMap_k(\elementaryE_\mu)\big)=\sum_{i=0}^{n-1} K_{\SH(i),\mu} \, t(t-1)^i.
\end{equation}
By \eqref{eq:phiIne},
\[
[t^r]\stanleyMap(\rowMap_k(\elementaryE_\mu))=\sum_{\substack{\nu\vdash n\ \\\ell(\nu)=r}} [\elementaryE_\nu]\rowMap_k(\elementaryE_\mu).
\]
Combining \eqref{eq:phiInrowMap} with the binomial expansion 
\[
t(t-1)^i=
\sum_{r=1}^{i+1} (-1)^{i-r+1}\binom{i}{r-1}t^r,
\]
we obtain an explicit alternating-binomial formula
\begin{equation}\label{eq:AlterFormula}
    \sum_{\substack{\nu\vdash n\\\ell(\nu)=r}} [\elementaryE_\nu]\rowMap_k(\elementaryE_\mu) 
    =\sum_{i=r-1}^{n-1}(-1)^{i-r+1}\binom{i}{r-1} K_{\SH(i),\mu}.
\end{equation}

We can give an unweighted interpretation of \eqref{eq:AlterFormula} using block-tableaux. For $T\in\SSYT(\SH(0),\mu)$, write its columns from left to right as $C_1,\ldots,C_n$. Each column is a strictly increasing block of $k$ entries. For a block $C$, write $\min C$ and $\max C$ for its first and last
entries. The row conditions on $T$ are
\begin{equation}\label{eq:block-word-condition}
 \min C_j\le\max C_{j+1}\qquad(1\le j<n).
\end{equation}
We call $j\in\{2,\ldots,n\}$ a \defin{record position} of $T$ if
\[
 \min C_j>\max\{\max C_1,\ldots,\max C_{j-1}\}.
\]
Let $\mathcal R(T)$ denote the set of record positions of $T$.

The first column of $\SH(i)$ has $k(i+1)$ cells. Each of the remaining $n-i-1$ columns has $k$ cells. The top cell of each column lies next to the bottom cell of the next column. This is the only row in which the two columns overlap. We construct a bijection between $\SSYT(\SH(i),\mu)$ and the pairs
\begin{equation}\label{eq:block-marked-pairs}
 (T,S),\qquad
 T\in\SSYT(\SH(0),\mu),\quad
 S\subseteq\mathcal R(T),\quad |S|=i.
\end{equation}

Given $S=\{j_1<\cdots<j_i\}$, place $C_1,C_{j_1},\ldots,C_{j_i}$ from top to bottom in the first column of $\SH(i)$. Place the remaining blocks in the other columns, keeping their order. The first column is strictly increasing because the selected blocks are records.

It remains to check the row conditions. Let $C_a$ and $C_b$ be consecutive unselected blocks, with $a<b$. If they are adjacent in $T$, the required inequality follows from \eqref{eq:block-word-condition}. Otherwise, let $C$ be the selected block immediately preceding $C_b$. By the record condition and \eqref{eq:block-word-condition}, $\min C_a\le \max C_a<\min C\le\max C_b$. The same argument applies to $C_1$ and the first unselected block, if any. Hence the resulting filling is semistandard.

For the inverse, take $P\in\SSYT(\SH(i),\mu)$. Split its first column from top to bottom into $k$-blocks $B_0,B_1,\ldots,B_i$. Let $D_1,\ldots,D_m$ be the other columns, in order from left to right, where $m=n-i-1$. For $1\le a\le i$, let
\[
 t_a=\min\{j\in[m]:\max D_j\ge\min B_a\}.
\]
If the set is empty, let $t_a=m+1$. Start with $B_0,D_1,\ldots,D_m$ and insert $B_a$ before $D_{t_a}$. If $t_a=m+1$, place $B_a$ at the end. When several blocks have the same insertion position, keep them in the order $B_1,\ldots,B_i$.

The first column of $P$ is strictly increasing, so $\max B_a<\min B_{a+1}$ and $t_1\le\cdots\le t_i$. Every $D$-block before $B_a$ has maximum less than $\min B_a$ by the definition of $t_a$. The same is true for every earlier $B$-block. Hence each inserted block is a record.

We check \eqref{eq:block-word-condition} for consecutive blocks in the new sequence. For two $D$-blocks, the condition follows from a row of $P$. For two $B$-blocks, it follows from the first column of $P$. If a $D$-block precedes $B_a$, its maximum is less than $\min B_a$. If $B_a$ precedes $D_j$, then $j=t_a$ and $\min B_a\le\max D_j$. Finally, if $B_0$ precedes $D_1$, the condition follows from a row of $P$.

These blocks form a tableau $T$ of shape $\SH(0)$. Let $S$ be the set of positions of the inserted blocks. Then $S$ consists of $i$ record positions.

We show that the maps are inverse. Start with $(T,S)$ and a marked block $C_{j_a}$. Every earlier unmarked block has maximum less than $\min C_{j_a}$ because $C_{j_a}$ is a record. If an unmarked block follows, the first one immediately follows either $C_{j_a}$ or a later marked block. In the latter case, the later block begins above every entry of $C_{j_a}$. The row condition in $T$ then shows that the maximum of the first following unmarked block is at least $\min C_{j_a}$. The inverse therefore restores $C_{j_a}$ to its original position. If no unmarked block follows, the inverse places it at the end. In the other direction, the forward map returns $B_0,\ldots,B_i$ and $D_1,\ldots,D_m$ to their original columns in $P$.

When $i=0$, no blocks are moved. When $i=n-1$, there are no $D$-blocks and the inverse reads the blocks of the first column in order. These conventions also cover $n=1$.

Figure~\ref{fig:blockTableau} illustrates the bijection for $k=2$,
$n=6$, and $\mu=(2,2,1^8)$. Both tableaux correspond to the same
$T\in\SSYT(\SH(0),\mu)$, whose columns are
\[
 (1,2),\ (1,2),\ (3,4),\ (5,6),\ (7,8),\ (9,10).
\]
The record positions of $T$ are $\{3,4,5,6\}$. The marked sets for the left and right tableaux are $\{3,4,6\}$ and $\{4\}$,
respectively. In the Boolean cancellation below, both tableaux
belong to the lattice with $H=\{4\}$. The block $(7,8)$ is an
unmarked record block and can also be promoted. The first two blocks have the same entries, but they have different positions in $T$.

\begin{figure}[!htbp]
\centering
\ytableausetup{boxsize=1em}
\begin{ytableau}
\none & \none & 7 \\
\none &1 &8 \\
1 & 2  \\
2 \\
*(pastelblue)3\\
*(pastelblue)4\\
5\\
6\\
*(pastelblue)9\\
*(pastelblue)10
\end{ytableau}
\quad
\begin{ytableau}
\none & \none &\none & \none &*(pastelblue)9\\
\none & \none &\none & 7&*(pastelblue)10\\
\none & \none &*(pastelblue)3&8\\
\none & 1&*(pastelblue)4\\
1&2\\
2\\
5\\
6
\end{ytableau}
\caption{Two tableaux with the same underlying $T$. The left tableau has shape $\SH(3)$ and the right one has shape $\SH(1)$. The blue blocks are removed from the first column when passing from left to right.}
\label{fig:blockTableau}
\end{figure}

\begin{definition}
A tableau $P\in\SSYT(\SH(r-1),\mu)$ is a
\defin{fixed-point block-tableau} if its corresponding pair $(T,S)$ in \eqref{eq:block-marked-pairs} satisfies $S=\mathcal R(T)$.
\end{definition}

Let $(T,S)$ be a pair in \eqref{eq:block-marked-pairs}, and $j\in\mathcal R(T)\setminus S$. We \defin{promote} the block $C_j$ by passing to the tableau corresponding to $(T,S\cup\{j\})$. This moves $C_j$ to the first column while keeping $T$ unchanged.

If $j,j'\in\mathcal R(T)\setminus S$ are distinct, then promoting $C_j$ and $C_{j'}$ in either order gives the tableau corresponding to $(T,S\cup\{j,j'\})$. Thus $P$ is fixed exactly when every record position is marked.

For example, let $k=2$ and $T$ have columns
$(1,3),(2,5),(4,6)$. Since $\mathcal R(T)=\varnothing$, the
tableau $T$ is fixed. Moving $(4,6)$ to the first column gives a
semistandard tableau, but the inverse map gives the underlying
tableau with columns $(1,3),(4,6),(2,5)$. This move is not a
promotion because it changes the underlying tableau.

\begin{theorem}\label{thm:fixed-block-tableaux}
Let $k,n\ge1$, $\mu\vdash kn$, and $1\le r\le n$. Then
\[
 \sum_{\substack{\nu\vdash n\\\ell(\nu)=r}}
 [\elementaryE_\nu]\rowMap_k(\elementaryE_\mu)
 =\bigl|\{P\in\SSYT(\SH(r-1),\mu):P\text{ is a fixed-point block-tableau}\}\bigr|.
\]
Equivalently, this number is the number of $T\in\SSYT(\SH(0),\mu)$ with $|\mathcal R(T)|=r-1$.
\end{theorem}
\begin{proof}
By the bijection above,
\[
 K_{\SH(i),\mu}
 =\sum_{T\in\SSYT(\SH(0),\mu)}\binom{|\mathcal R(T)|}{i}.
\]
For each $S$ of size $i$, the factor $\binom{i}{r-1}$ in \eqref{eq:AlterFormula} counts the subsets $H\subseteq S$ of size $r-1$. The right side is therefore a signed count of triples $(T,S,H)$ with $H\subseteq S\subseteq\mathcal R(T)$ and $|H|=r-1$. Each triple has sign $(-1)^{|S|-|H|}$. Fix $T$ and $H$. If $H\ne\mathcal R(T)$, let $j$ be the smallest
element of $\mathcal R(T)\setminus H$. Replacing $S$ by
$S\mathbin{\triangle}\{j\}$ is a sign-reversing involution on this class.
Equivalently, the possible sets $S\setminus H$ are all subsets of $\mathcal R(T)\setminus H$, so their total contribution is
\[
 \sum_{H\subseteq S\subseteq\mathcal R(T)}(-1)^{|S|-|H|}
 =\begin{cases}1,&H=\mathcal R(T),\\0,&H\ne\mathcal R(T).\end{cases}
\]
Only the triples with $H=\mathcal R(T)$ survive. For these triples, $|\mathcal R(T)|=r-1$. By the bijection with $i=r-1$, they correspond to the fixed-point block-tableaux.
\end{proof}
For $r=1$, we take $H=\varnothing$. For $r=n$, the only possible surviving set is $\mathcal R(T)=\{2,\ldots,n\}$. When $k=1$, the blocks form a weakly increasing word, so $|\mathcal R(T)|+1=\ell(\mu)$, as required by $\rowMap_1(\elementaryE_\mu)=\elementaryE_\mu$.

Figure~\ref{fig:booleanPosetExample} shows the Boolean lattice for $k=2$, $n=3$, $r=1$, $\mu=1^6$, and the tableau $T$ with columns $(1,2)$, $(3,4)$, and $(5,6)$. Here $H=\varnothing$ and $\mathcal R(T)=\{2,3\}$. The top, left, right, and bottom vertices correspond to $S=\varnothing$, $\{3\}$, $\{2\}$, and $\{2,3\}$, respectively. Their signs are $1,-1,-1,1$. The bottom vertex has $S=\mathcal R(T)$, but the class cancels because $H\ne\mathcal R(T)$; a surviving class requires $H=\mathcal R(T)$.
\begin{figure}[htbp]
\centering
\begin{tikzpicture}[scale=1.1, every node/.style={inner sep=2pt}]
\node (top) at (0,1.5) {$\begin{ytableau}
\none&\none&*(pastelblue)5\\
\none&*(pastelblue)3&*(pastelblue)6\\
1&*(pastelblue)4\\
2
\end{ytableau}$};
\node (left) at (-1.5,0) {$\begin{ytableau}
\none&*(pastelblue)3\\
1&*(pastelblue)4\\
2\\
*(pastelblue)5\\
*(pastelblue)6
\end{ytableau}$};
\node (right) at (1.5,0) {$\begin{ytableau}
\none&*(pastelblue)5\\
1&*(pastelblue)6\\
2\\
*(pastelblue)3\\
*(pastelblue)4
\end{ytableau}$};
\node (bot) at (0,-1.5) {$P=\begin{ytableau}
1\\
2\\
*(pastelblue)3\\
*(pastelblue)4\\
*(pastelblue)5\\
*(pastelblue)6
\end{ytableau}$};
\draw (bot)--(left)--(top)--(right)--(bot);
\end{tikzpicture}
\caption{The Boolean lattice for $T$ and $H=\varnothing$. Blue marks the two record blocks in every tableau; a record belongs to $S$ exactly when its block lies in the first column. Each edge adds or removes one record from $S$.}
\label{fig:booleanPosetExample}
\end{figure}

\begin{samepage}
There is also an interval description. Let $I_j=[\min C_j,\max C_j]$. Read the intervals from left to right, starting a new segment when the next interval is disjoint from the current span. By \eqref{eq:block-word-condition}, it must lie to the right. The maximum of the current span is always the maximum of all entries read so far. Thus the cuts occur at the record positions, and the number of segments is $|\mathcal R(T)|+1$.\par
\end{samepage}

\begin{example}
Let $k=2$, $n=4$, and $\mu=(2,1^6)$. Then
\[
 \rowMap_2(\elementaryE_\mu)
 =6\elementaryE_{211}+8\elementaryE_{22}
  +104\elementaryE_{31}+544\elementaryE_4.
\]
The sum of the $\elementaryE$-coefficients indexed by partitions of length $2$ is $8+104=112$. We also have
\[
 K_{\SH(1),\mu}=124,\qquad K_{\SH(2),\mu}=6,
 \qquad K_{\SH(3),\mu}=0.
\]
Thus \eqref{eq:AlterFormula} gives $124-2\cdot6=112$.

There are six tableaux $T\in\SSYT(\SH(0),\mu)$ with two record positions, and none with three. Each of these six has two choices of $H$ of size $1$, giving $12$ Boolean classes of rank $1$.

Each class contains one non-fixed tableau of shape $\SH(1)$ and one tableau of shape $\SH(2)$. The latter may occur in two different classes, with different choices of $H$. Figure~\ref{fig:mu2111111-r2} shows the $12$ non-fixed tableaux of shape $\SH(1)$, with the unique unmarked record block in blue. The remaining $124-12=112$ tableaux of shape $\SH(1)$ are fixed.
\end{example}

\begin{figure}[!htbp]
\centering
\ytableausetup{boxsize=1.0em}
$
\begin{array}{cccccc}
\ytableaushort{{\none}{\none}{*(pastelblue)6},{\none}1{*(pastelblue)7},15{\none},2{\none}{\none},3{\none}{\none},4{\none}{\none}}
&
\ytableaushort{{\none}{\none}1,{\none}{*(pastelblue)5}7,1{*(pastelblue)6}{\none},2{\none}{\none},3{\none}{\none},4{\none}{\none}}
&
\ytableaushort{{\none}{\none}1,{\none}{*(pastelblue)5}6,1{*(pastelblue)7}{\none},2{\none}{\none},3{\none}{\none},4{\none}{\none}}
&
\ytableaushort{{\none}{\none}{*(pastelblue)6},{\none}1{*(pastelblue)7},14{\none},2{\none}{\none},3{\none}{\none},5{\none}{\none}}
&
\ytableaushort{{\none}{\none}{*(pastelblue)6},{\none}1{*(pastelblue)7},13{\none},2{\none}{\none},4{\none}{\none},5{\none}{\none}}
&
\ytableaushort{{\none}{\none}1,{\none}{*(pastelblue)3}7,1{*(pastelblue)4}{\none},2{\none}{\none},5{\none}{\none},6{\none}{\none}}
\\[0.9em]
\ytableaushort{{\none}{\none}1,{\none}{*(pastelblue)3}6,1{*(pastelblue)4}{\none},2{\none}{\none},5{\none}{\none},7{\none}{\none}}
&
\ytableaushort{{\none}{\none}{*(pastelblue)4},{\none}1{*(pastelblue)5},13{\none},2{\none}{\none},6{\none}{\none},7{\none}{\none}}
&
\ytableaushort{{\none}{\none}1,{\none}{*(pastelblue)3}5,1{*(pastelblue)4}{\none},2{\none}{\none},6{\none}{\none},7{\none}{\none}}
&
\ytableaushort{{\none}{\none}1,{\none}{*(pastelblue)3}4,1{*(pastelblue)5}{\none},2{\none}{\none},6{\none}{\none},7{\none}{\none}}
&
\ytableaushort{{\none}{\none}{*(pastelblue)6},{\none}1{*(pastelblue)7},12{\none},3{\none}{\none},4{\none}{\none},5{\none}{\none}}
&
\ytableaushort{{\none}{\none}{*(pastelblue)4},{\none}1{*(pastelblue)5},12{\none},3{\none}{\none},6{\none}{\none},7{\none}{\none}}
\end{array}
$
\caption{The $12$ non-fixed tableaux of shape $\SH(1)$ and content $(2,1^6)$. In each tableau, the blue block is the unique unmarked record block; promoting it gives a tableau of shape $\SH(2)$.}
\label{fig:mu2111111-r2}
\end{figure}

\subsection{A stable two-part expansion}

The preceding theorem also gives an elementary-positive expansion for a
family whose coefficients stabilize as the parts increase.

\begin{lemma}\label{lemma:two-row-support}
For each $k\ge1$, there is a unique sequence of integers
$A_0^{(k)},A_1^{(k)},\ldots$, independent of $c$, with $A_0^{(k)}=1$, such that
\begin{equation}\label{eq:stable-two-part}
 \rowMap_k(\elementaryE_c^{2k})
 =\sum_{t=0}^{c}A_t^{(k)}
   \elementaryE_{c+t}\elementaryE_{c-t}
 \qquad(c\ge0).
\end{equation}
\end{lemma}
\begin{proof}
Every exponent in $\elementaryE_c^{2k}$ is at most $2k$.
After applying $\rowMap_k$, every exponent is at most $2$ and
the degree is $2c$. Thus the monomial expansion is supported on
$2^{c-d}1^{2d}$ for $0\le d\le c$.

The coefficient of $\monomial_{2^{c-d}1^{2d}}$ equals
the coefficient of
$x_1^{2k}\cdots x_{c-d}^{2k}x_{c-d+1}^k\cdots x_{c+d}^k$
in $\elementaryE_c^{2k}$.
Each of the first $c-d$ variables must occur in all $2k$ factors.
Removing them from each factor leaves
\[
 N_d^{(k)}
 \coloneqq [x_1^k\cdots x_{2d}^k]
 \elementaryE_d(x_1,\ldots,x_{2d})^{2k},
\]
which does not depend on $c$. We take $N_0^{(k)}=1$.
Therefore
\[
 \rowMap_k(\elementaryE_c^{2k})
 =\sum_{d=0}^{c}N_d^{(k)}
   \monomial_{2^{c-d}1^{2d}}.
\]
For $0\le t\le c$, choosing which of the $2d$ variables
occur in only the second factor gives
\[
 \elementaryE_{c+t}\elementaryE_{c-t}
 =\sum_{d=t}^{c}\binom{2d}{d-t}
   \monomial_{2^{c-d}1^{2d}}.
\]
This transition matrix is unitriangular. Its coefficients do not
depend on $c$, so the equations
\[
 A_0^{(k)}=1,\qquad
 A_d^{(k)}=N_d^{(k)}
   -\sum_{t=0}^{d-1}\binom{2d}{d-t}A_t^{(k)}
 \quad(d\ge1)
\]
define a unique integer sequence satisfying
\eqref{eq:stable-two-part} for every $c$.
\end{proof}

\begin{corollary}\label{cor:two-part-positive}
For every $k\ge1$ and $t\ge0$, the coefficient $A_t^{(k)}$
is nonnegative. In particular, $\rowMap_k(\elementaryE_c^{2k})$
is $\elementaryE$-positive for every $c\ge0$.
\end{corollary}
\begin{proof}
We have $A_0^{(k)}=1$.
For $t\ge1$, apply \eqref{eq:stable-two-part} with $c=t$.
Only the summand indexed by $t$ has an elementary index of length
one, so
\[
 A_t^{(k)}
 =[\elementaryE_{2t}]\rowMap_k(\elementaryE_t^{2k}).
\]
Theorem~\ref{thm:fixed-block-tableaux}, with $n=2t$,
$\mu=(t^{2k})$, and $r=1$, expresses this coefficient as the
number of fixed-point block-tableaux of shape $\SH(0)$ and
content $(t^{2k})$. Hence $A_t^{(k)}\ge0$.
\end{proof}

\subsection{Two-variable specialization}\label{sec:two-variable}

Let $F_{n,k}=\rowMap_k(\elementaryE_1^{kn})$. Its specialization to two variables is
\[
 F_{n,k}(x,y)=\sum_{a=0}^{n}\binom{kn}{ka}x^ay^{n-a}.
\]
Since $\elementaryE_\lambda(x,y)=0$ when $\lambda_1>2$, this specialization determines the coefficients indexed by partitions with parts at most $2$. Write
\[
 \gamma^{(k)}_{n,d}
 = [\elementaryE_{2^d1^{n-2d}}]F_{n,k}
 \qquad (0\le d\le\lfloor n/2\rfloor).
\]
Then
\begin{equation}\label{eq:gamma-two-variable}
 F_{n,k}(x,y)
 =\sum_{d=0}^{\lfloor n/2\rfloor}
   \gamma^{(k)}_{n,d}(xy)^d(x+y)^{n-2d}.
\end{equation}
Equivalently, the polynomial
\[
 P_{n,k}(z)=\sum_{m=0}^{n}\binom{kn}{km}z^m
\]
has the expansion
\begin{equation}\label{eq:gamma-expansion}
 P_{n,k}(z)
 =\sum_{d=0}^{\lfloor n/2\rfloor}
   \gamma^{(k)}_{n,d}z^d(1+z)^{n-2d}.
\end{equation}
Comparing coefficients gives, for $0\le m\le n$,
\begin{equation}\label{eq:coefdicofz^m}
 \binom{kn}{km}
 =\sum_{d=0}^{\lfloor n/2\rfloor}
   \gamma^{(k)}_{n,d}\binom{n-2d}{m-d},
\end{equation}
where a binomial coefficient is zero when its lower index is outside the usual range.

\begin{proposition}\label{prop:gamma-nonnegative}
For $k,n\ge1$ and $0\le d\le\lfloor n/2\rfloor$, the coefficient $\gamma^{(k)}_{n,d}$ is nonnegative.
\end{proposition}
\begin{proof}
The polynomial $P_{n,k}(z)$ is the zeroth Veronese $k$-section of $(1+z)^{kn}$, obtained by retaining the coefficients whose indices are divisible by $k$ and dividing their exponents by $k$. Since $(1+z)^{kn}$ is real-rooted with nonnegative coefficients, its Veronese sections are real-rooted; see \cite[Theorem~5.3 and Remark~5.4]{AthanasiadisWagner2024}. Thus all roots of $P_{n,k}(z)$ are real. They are negative because its coefficients are positive and its constant term is $1$.

Moreover, $\binom{kn}{km}=\binom{kn}{k(n-m)}$, so $P_{n,k}(z)$ is palindromic. Its negative roots therefore occur in reciprocal pairs $-a,-a^{-1}$ with $a>0$, apart from a possible unpaired root at $-1$. Each paired factor is
\[
 (z+a)(z+a^{-1})=(1+z)^2+(a+a^{-1}-2)z,
\]
where $a+a^{-1}-2\ge0$. The unpaired factor, if present, is $1+z$. Multiplying these factors gives nonnegative coefficients in the basis $z^d(1+z)^{n-2d}$, which proves the claim. See also \cite[Lemma~4.1]{Branden2004sg}.
\end{proof}

For $k=2$, these coefficients have a simple formula and a direct interpretation.

\begin{proposition}\label{cor:sh0-explicit-12n}
For every $0\le d\le\lfloor n/2\rfloor$,
\[
 [\elementaryE_{2^d1^{n-2d}}]\rowMap_2(\elementaryE_1^{2n})
 =4^d\binom{n}{2d}.
\]
\end{proposition}
\begin{proof}
Let $s=x+y$ and $p=xy$. The even part of the binomial expansion gives
\begin{align*}
 F_{n,2}(x,y)
 =\frac{(\sqrt{x}+\sqrt{y})^{2n}
          +(\sqrt{x}-\sqrt{y})^{2n}}{2}
 =\frac{(s+2\sqrt{p})^n+(s-2\sqrt{p})^n}{2}
 =\sum_{d=0}^{\lfloor n/2\rfloor}
   4^d\binom{n}{2d}p^ds^{n-2d}.
\end{align*}
Compare this with \eqref{eq:gamma-two-variable}.
\end{proof}

\begin{remark}\label{rem:gamma-block-orbits}
Partition a word in $\{E,N\}^{2n}$ into $n$ consecutive blocks of length $2$. A block is mixed if it is $EN$ or $NE$, and pure if it is $EE$ or $NN$. Let $B_{n,d}$ be the set of such words with exactly $2d$ mixed blocks. Every word in $B_{n,d}$ has an even number of each letter.

There is an action of $(\mathbb Z/2\mathbb Z)^{n-2d}$ on $B_{n,d}$, where the $j$th generator interchanges $EE$ and $NN$ in the $j$th pure block, read from left to right. An orbit is determined by its $2d$ mixed positions and the choice of $EN$ or $NE$ at each of them. Hence there are
\[
 \binom{n}{2d}2^{2d}=\gamma^{(2)}_{n,d}
\]
orbits. Within each orbit, the number of words with $2m$ occurrences of $E$ is $\binom{n-2d}{m-d}$, because the mixed blocks contribute $2d$ occurrences of $E$ and exactly $m-d$ pure blocks must be $EE$. This gives the identity
\begin{equation}\label{eq:gammaKIs2}
 \sum_{m=0}^{n}\binom{2n}{2m}z^m
 =\sum_{d=0}^{\lfloor n/2\rfloor}
   4^d\binom{n}{2d}z^d(1+z)^{n-2d}.
\end{equation}
\end{remark}

The general elementary-positivity theorem in
\cite{AlexanderssonDai2026Epositive} gives a formula with positive rational
weights for the individual coefficients. The block-tableau formula above
is unweighted after summing over partitions of a fixed length.
It remains to find a direct unweighted interpretation of each coefficient
$[\elementaryE_\nu]\rowMap_k(\elementaryE_\mu)$.

\section{Applications}\label{Application}
\subsection{Hadamard products of Jacobi--Trudi matrices}

\subsubsection{The adjoint row map of Schur functions}

Let $H(\lambda)$ and $H(\lambda/\mu)$ denote the following matrices 
of complete homogeneous symmetric functions:
\begin{equation}\label{eq:jacobiTrudiMatrices}
H(\lambda)_{ij} \coloneqq  \completeH_{\lambda_i - i + j}(\xvec),
\qquad 
H(\lambda/\mu)_{ij} \coloneqq  \completeH_{\lambda_i - \mu_j - i + j}(\xvec).
\end{equation}
The Jacobi--Trudi identity (see \cite{Macdonald1995}) gives 
\[
 \schur_\lambda = \det H(\lambda), \qquad 
 \schur_{\lambda/\mu} = \det H(\lambda/\mu).
\]

\begin{definition}
Let $A=\left(\completeH_{a_{ij}}\right)$ and $B=\left(\completeH_{b_{ij}}\right)$ be matrices with their integer indices recorded. For matrices of the same dimension, define $A\star B$ by adding the indices before evaluating the complete symmetric functions:
\[
\defin{(A \star  B)_{ij}}\coloneqq \completeH_{a_{ij}+b_{ij}}.
\]
\end{definition}
In particular, the indices of entries with negative subscripts must be retained even though those entries evaluate to zero.

\begin{theorem}
Let $\lambda^{(1)}, \dotsc, \lambda^{(k)}$ be partitions, all extended to have the same length $\ell$. Let $\rho = (\ell-1, \ell-2, \dotsc, 1, 0)$ 
be the staircase partition. Then 
\begin{equation}\label{eq:starProductDet}
\det\!\bigl(H(\lambda^{(1)}) \star H(\lambda^{(2)}) \star \dotsb \star H(\lambda^{(k)})\bigr) 
=
\schur_{ (\lambda^{(1)}+\cdots+\lambda^{(k)}+(k-1)\rho)/(k-1)\rho }.
\end{equation}
\end{theorem}
\begin{proof}
By the definition of $\star$ and the expressions in 
\eqref{eq:jacobiTrudiMatrices}, the entry indexed $(i,j)$ in the matrix in \eqref{eq:starProductDet}
is given by the complete homogeneous symmetric function
indexed by the sum
\[
\sum_{t=1}^k \left( \lambda^{(t)}_i - i + j\right)
= \left(\sum_{t=1}^k \lambda^{(t)}_i\right) - k i + k j.
\]
If we now let
\[
\nu \coloneqq \sum_{t=1}^k \lambda^{(t)} + (k-1)\rho,
\quad 
\tau \coloneqq (k-1)\rho,
\]
then for all $1\leq i,j\leq \ell$,
\[
\nu_i - \tau_j - i + j
= \left(\sum_t \lambda^{(t)}_i + (k-1)(\ell - i)\right)- (k-1)(\ell - j)-i + j = 
\left(\sum_{t=1}^k \lambda^{(t)}_i \right) - k i + k j.
\]
Therefore,
 \[
 H(\lambda^{(1)}) \star H(\lambda^{(2)}) \star \dotsb \star 
H(\lambda^{(k)}) = H(\nu / \tau )= 
  H\left( 
  (\lambda^{(1)}+\cdots+\lambda^{(k)} + (k-1)\rho) / (k-1)\rho \right).
\]
Taking determinants gives \eqref{eq:starProductDet}.
\end{proof}

\subsubsection{Hadamard products of Jacobi--Trudi matrices}

\begin{definition}[Hadamard product]
    For two matrices $A=(x_{ij})$ and $B=(y_{ij})$ of the same dimension $m \times n$, the \defin{Hadamard product} $A * B$ is the
    matrix given by element-wise product,
    \defin{$(A*B)_{ij} =x_{ij}y_{ij}$.}
\end{definition}

The Schur expansion coefficients of $\colMap_k(\schur_{\lambda^k})$ are given by
\[
[\schur_\mu] \colMap_k(\schur_{\lambda^k})
= \langle \schur_{\lambda^k}, \colMap_k^\perp(\schur_\mu) \rangle
= \left\langle \schur_{\lambda^k}, \det\left(\completeH_{(\mu_i - i + j)^k}\right) \right\rangle.
\]

Thus, the adjoint operator $\colMap_k^\perp$ 
corresponds to taking Hadamard powers of Jacobi--Trudi matrices:
\[
\colMap_k^\perp(\schur_\lambda) =
\det \; \left( \completeH_{\lambda_i - i + j} \right)^{*k}.
\]

\begin{remark}
  So for $\lambda^{(1)} = \lambda^{(2)} = (2,2,2,2)$,
  the Hadamard product above is not even monomial positive, with 
  $(-4)\monomial_{10,3,3}$ appearing in the monomial expansion.
Direct computation gives that 
\[
 [\schur_{10,3,3}] \colAdj{2}(\schur_{(2,2,2,2)}) = -4,
\]
and $\colMap_2(\schur_{10,3,3})$ is thus not Schur positive.
\end{remark}

In fact, $\colMap_2(\schur_{\lambda^2})$ is \emph{not} Schur positive in general.
The smallest counterexample occurs at $\lambda^2 = (10,10,2,2)$:
\begin{align*}
\colMap_2(\schur_{(10,10,2,2)})
&= 54200\,\schur_{1^{12}}
+ 58880\,\schur_{2,1^{10}}
+ 34882\,\schur_{2^2,1^8}
+ 12218\,\schur_{2^3,1^6}
\\&\quad
+ 1490\,\schur_{2^4,1^4}
+ 282\,\schur_{2^5,1^2}
+ 40\,\schur_{2^6}
+ 36282\,\schur_{3,1^9}
\\&\quad
+ 27784\,\schur_{3,2,1^7}
+ 10502\,\schur_{3,2^2,1^5}
+ 1490\,\schur_{3,2^3,1^3}
+ 294\,\schur_{3,2^4,1}
\\&\quad
+ 9436\,\schur_{3^2,1^6}
+ 4614\,\schur_{3^2,2,1^4}
+ 802\,\schur_{3^2,2^2,1^2}
+ 158\,\schur_{3^2,2^3}
\\&\quad
+ 988\,\schur_{3^3,1^3}
+ 180\,\schur_{3^3,2,1}
+ \mathbf{(-14)}\,\schur_{3^4}
+ 18450\,\schur_{4,1^8}
\\&\quad
+ 13726\,\schur_{4,2,1^6}
+ 5072\,\schur_{4,2^2,1^4}
+ 906\,\schur_{4,2^3,1^2}
+ 174\,\schur_{4,2^4}
\\&\quad
+ 5630\,\schur_{4,3,1^5}
+ 2668\,\schur_{4,3,2,1^3}
+ 586\,\schur_{4,3,2^2,1}
+ 548\,\schur_{4,3^2,1^2}
\\&\quad
+ 106\,\schur_{4,3^2,2}
+ 1328\,\schur_{4^2,1^4}
+ 648\,\schur_{4^2,2,1^2}
+ 148\,\schur_{4^2,2^2}
\\&\quad
+ 146\,\schur_{4^2,3,1}
+ 18\,\schur_{4^3}
+ 8330\,\schur_{5,1^7}
+ 5546\,\schur_{5,2,1^5}
\\&\quad
+ 1984\,\schur_{5,2^2,1^3}
+ 454\,\schur_{5,2^3,1}
+ 1996\,\schur_{5,3,1^4}
+ 926\,\schur_{5,3,2,1^2}
\\&\quad
+ 216\,\schur_{5,3,2^2}
+ 156\,\schur_{5,3^2,1}
+ 514\,\schur_{5,4,1^3}
+ 250\,\schur_{5,4,2,1}
\\&\quad
+ 36\,\schur_{5,4,3}
+ 75\,\schur_{5^2,1^2}
+ 31\,\schur_{5^2,2}
+ 3202\,\schur_{6,1^6}
\\&\quad
+ 1866\,\schur_{6,2,1^4}
+ 668\,\schur_{6,2^2,1^2}
+ 160\,\schur_{6,2^3}
+ 532\,\schur_{6,3,1^3}
\\&\quad
+ 250\,\schur_{6,3,2,1}
+ 28\,\schur_{6,3^2}
+ 105\,\schur_{6,4,1^2}
+ 45\,\schur_{6,4,2}
\\&\quad
+ 14\,\schur_{6,5,1}
+ \schur_{6^2}
+ 982\,\schur_{7,1^5}
+ 494\,\schur_{7,2,1^3}
\\&\quad
+ 184\,\schur_{7,2^2,1}
+ 105\,\schur_{7,3,1^2}
+ 43\,\schur_{7,3,2}
+ 14\,\schur_{7,4,1}
\\&\quad
+ \schur_{7,5}
+ 218\,\schur_{8,1^4}
+ 95\,\schur_{8,2,1^2}
+ 33\,\schur_{8,2^2}
\\&\quad
+ 14\,\schur_{8,3,1}
+ \schur_{8,4}
+ 29\,\schur_{9,1^3}
+ 12\,\schur_{9,2,1}
\\&\quad
+ \schur_{9,3}
+ \schur_{10,1^2}
+ \schur_{10,2}.
\end{align*}
The unique negative coefficient $\mathbf{-14}$ appears at $\schur_{(3,3,3,3)}$.
Similarly, $\colMap_2(\schur_{(9,9,2,2,1,1)})$ has coefficient $-20$ at $\schur_{(3,3,3,3)}$.

\subsection{Comparison with the Verschiebung operator}

\begin{definition}[The Adams operator and the Verschiebung operator]
The \defin{$k$th Adams operator} on symmetric functions is the ring homomorphism 
\[
 f(\xvec) \mapsto f \circ \powersum_k = f(x_1^k, x_2^k, \dotsc ).
\]
The adjoint operator is the \defin{Verschiebung operator} $\versch_k$, which is defined as
\[
 \versch_k \powersum_m(\xvec) = 
 \begin{cases}
k\cdot \powersum_{m/k}(\xvec) & \text{ if } k \mid m \\
 0 & \text{ otherwise,}
 \end{cases}
\]
and extended multiplicatively. 
\end{definition}

\begin{lemma}
The Verschiebung operator acts as follows:
\[
 \versch_k(\completeH_m) = 
 \begin{cases} 
 \completeH_{m/k} & \text{if } k \mid m \\ 0 &\text{otherwise}
 \end{cases}
 \qquad 
 \text{ and }  
 \qquad 
 \versch_k(\elementaryE_m) =
  \begin{cases} 
 (-1)^{\frac{m}{k}(k-1)}\elementaryE_{m/k} & \text{if } k \mid m \\ 0 &\text{otherwise}.
\end{cases}
\]
\end{lemma}

\begin{corollary}\label{cor:versch}
Let $\lambda \vdash n$ and $k$ be a positive integer.
Define the coefficients $A_{\lambda,\mu}$, $B_{\lambda,\mu}$, and $C_{\lambda,\mu}$ via the relations
\[
  \schur_{k\lambda}(\xvec) = \sum_{\mu \vdash kn} A_{\lambda,\mu} \completeH_{\mu}(\xvec)
  = \sum_{\mu \vdash kn} B_{\lambda,\mu} \elementaryE_{\mu}(\xvec)
  = \sum_{\mu \vdash kn} C_{\lambda,\mu} \frac{\powersum_{\mu}(\xvec)}{z_\mu}.
\]
Then
\[
\versch_k( \schur_{k\lambda}(\xvec) ) = \sum_{\mu \vdash n} A_{\lambda,k\mu} \completeH_{\mu}(\xvec)
  = (-1)^{n(k-1)} \sum_{\mu \vdash n} B_{\lambda,k\mu} \elementaryE_{\mu}(\xvec)
  = \sum_{\mu \vdash n} C_{\lambda,k\mu} \frac{\powersum_{\mu}(\xvec)}{z_\mu}.
\]
\end{corollary}
\begin{proof}
The identities for the complete homogeneous and elementary symmetric functions follow directly after applying $\versch_k$.

We now consider the power-sum expansion.
By the definition of $z_\mu$, scaling every part of $\mu$ by $k$ gives
\[
z_{k\mu} = \prod_i (ki)^{m_i} m_i! = k^{\sum m_i} \prod_i i^{m_i} m_i! = k^{\ell(\mu)} z_\mu.
\]
The factor $k^{\ell(\mu)}$ cancels the extra factors of $k$ from applying $\versch_k$ to $\powersum_{k\mu}$.
\end{proof}

We use the following result of Albion~\cite{Albion2022} on the Verschiebung operator.
\begin{theorem}\label{thm:albion}
 We have $\versch_k(\schur_{\lambda/\mu}) = 0$ 
 unless $\lambda/\mu$ is tileable by $k$-ribbons, in which case
\[
\versch_k\left( \schur_{\lambda/\mu} \right) = \operatorname{sgn}_k(\lambda/\mu) \prod_{r=0}^{k-1} \schur_{\lambda^{(r)}/\mu^{(r)}}.
\]
\end{theorem}
Here, $\operatorname{sgn}_k(\lambda/\mu)$ is the sign of (any) border-strip tableau of shape $\lambda/\mu$ and all ribbons have size $k$.
The skew shapes $\left\{\lambda^{(r)}/\mu^{(r)} \right\}_{r=0}^{k-1}$ in the product are given by the \emph{$k$-quotient} of $\lambda/\mu$. See \cite{AlexanderssonPfannererRubeyUhlin2020} for details.

\subsection{Schur positivity under Verschiebung}
In this subsection, $\lambda^{(0)},\ldots,\lambda^{(k-1)}$ denote
the components of the $k$-quotient of $k\lambda$.

Recall that the Schur functions expanded in the power-sum basis
are given by
\[
 \schur_{\lambda}(\xvec) = \sum_{\mu} \frac{\chi_{\lambda \mu}}{z_\mu} \powersum_\mu(\xvec),
\]
where $\chi_{\lambda \mu}$ are the $\symS_n$-character values. 
From Corollary~\ref{cor:versch}, it follows that 
\begin{equation}
   \versch_k(\schur_{k\lambda}) =  \sum_{\mu} \chi_{k\lambda,k\mu} \frac{\powersum_\mu(\xvec)}{z_\mu}.
\end{equation}
Moreover, since the shape $k\lambda$ is tileable by $|\lambda|$ horizontally arranged $k$-ribbons, it follows from Theorem~\ref{thm:albion} that
\[
\versch_k(\schur_{k\lambda}) = \prod_{r=0}^{k-1} \schur_{\lambda^{(r)}}
\]
which is Schur-positive as the right hand side is a product of Schur functions. 
In fact, $\versch_k(\schur_{k\lambda/k\mu})$ is Schur-positive via the same argument.

\begin{proposition}
 Suppose $k \geq \length(\lambda)$. 
 Then $\versch_k(\schur_{k\lambda})  = \completeH_\lambda$.
\end{proposition}
\begin{proof}
The definition of partition quotients gives (see \cite[p. 12]{Macdonald1995})
\[
\{\lambda^{(r)}:0\le r<k\}=\{(\lambda_j: j\equiv r\pmod{k}):0\le r<k\}
\]
as a multiset of partitions. The runner labels may be permuted by the choice of abacus origin.
If $k \geq \length(\lambda)$, then each $\lambda^{(r)}$ is a single part (or empty),
so their product is $\prod_j\completeH_{\lambda_j}=\completeH_\lambda$.
\end{proof}

\begin{example}
For $k=2$, we have:
\begin{align*}
\versch_k(\schur_{k\cdot 5}) &= \schur_5 \\
\versch_k(\schur_{k\cdot 41}) &= \schur_5+\schur_{41} \\
\versch_k(\schur_{k\cdot 32}) &= \schur_5+\schur_{32}+\schur_{41} \\
\versch_k(\schur_{k\cdot 311}) &= \schur_{32}+\schur_{41}+\schur_{311} \\
\versch_k(\schur_{k\cdot 221}) &= \schur_{32}+\schur_{41}+\schur_{221}+\schur_{311} \\
\versch_k(\schur_{k\cdot 2111}) &= \schur_{32}+\schur_{221}+\schur_{311}+\schur_{2111} \\
\versch_k(\schur_{k\cdot 11111}) &= \schur_{221}+\schur_{2111}+\schur_{11111}
\end{align*}
Note that for larger partitions, some of the coefficients are greater than $1$.
\end{example}

\subsection{Chromatic symmetric functions}
Applying $\rowMap_k$ to a chromatic symmetric function gives a generating function for ``fractional'' proper colorings. We consider the sum of its $\elementaryE$-coefficients.
Let
\[
\eta:\Lambda\to \mathbb Q,\qquad \eta(\elementaryE_r)=1\ \text{for all }r\ge 1,
\]
extended as a ring homomorphism. Thus for any symmetric function $F=\sum_\lambda c_\lambda \elementaryE_\lambda$, $\eta(F)$ is the sum of its $\elementaryE$-coefficients. On the integral symmetric functions, $\eta$ takes integer values.
Define the linear map $\Phi_k\colon\Lambda\to\QQ$ by $\Phi_k(F)=\eta(\rowMap_k(F))$.

For a graph $G=(V,E)$, write $n=|V|$. A standard expansion of the chromatic symmetric function is
\[
X_G=\sum_{\pi\in \mathrm{StPar}(G)} \widetilde \monomial_{\lambda(\pi)},
\]
where $\pi$ runs over stable set partitions of $n$, $\lambda(\pi)$ is the block-size partition, and
\[
\widetilde \monomial_\lambda=\Bigl(\prod_i m_i(\lambda)!\Bigr)\monomial_\lambda
\]
is the augmented monomial symmetric function, see~\cite{Stanley1995}.

Let $\mathrm{StPar}_k(G)$ be the set of stable partitions of $G$
whose block sizes are divisible by $k$. Applying $\rowMap_k$ to
$X_G$ gives
\[
 \rowMap_k(X_G)
 =\sum_{\pi\in\mathrm{StPar}_k(G)}
 \widetilde{\monomial}_{\lambda(\pi)/k}.
\]
Since
\[
 \eta(\widetilde{\monomial}_\mu)
 =(-1)^{|\mu|-\ell(\mu)}\ell(\mu)!,
\]
the sum of the $\elementaryE$-coefficients is
\[
 \Phi_k(X_G)
 =\sum_\lambda[\elementaryE_\lambda]\rowMap_k(X_G)
 =\sum_{\pi\in\mathrm{StPar}_k(G)}
 (-1)^{|\lambda(\pi)|/k-\ell(\lambda(\pi))}
 \ell(\lambda(\pi))!.
\]
This sum is zero if $k\nmid n$.

For a composition $\alpha=(\alpha_1,\ldots,\alpha_r)\vDash n$,
the monomial quasisymmetric function is
\[
 M_\alpha(\xvec)
 =\sum_{i_1<\cdots<i_r}
 x_{i_1}^{\alpha_1}\cdots x_{i_r}^{\alpha_r}.
\]
Define a linear map $\widehat{\Phi}_k$ on this basis by
\[
 \defin{\widehat{\Phi}_k(M_\alpha(\xvec))}\coloneqq
 \begin{cases}
  (-1)^{n/k-r},&\text{if }k\mid\alpha_i\text{ for every }i,\\
  0,&\text{otherwise}.
 \end{cases}
\]
For the empty composition, let $\widehat{\Phi}_k(1)=1$.
For a partition $\lambda$, we have
\[
 \widetilde{\monomial}_\lambda
 =\left(\prod_i m_i(\lambda)!\right)
 \sum_{\operatorname{sort}(\alpha)=\lambda}M_\alpha,
\]
where the sum runs over all rearrangements of $\lambda$ into compositions. Counting multiplicities gives
\[
\widehat{\Phi}_k(\widetilde{\monomial}_\lambda)=
\begin{cases}
(-1)^{|\lambda|/k - \ell(\lambda)} \ell(\lambda)!,
& \text{if all parts of } \lambda \text{ are divisible by } k,\\
0, & \text{otherwise}.
\end{cases}
\]
Hence 
$\widehat{\Phi}_k = \Phi_k$ 
on symmetric functions.
\begin{lemma}\label{lemm:Phi_kF}
Let $n = km$, and \defin{$T_{n,k}$} $\coloneqq \{k,2k,\dots,(m-1)k\}$. Then
\[
\widehat{\Phi}_k(F_{n,S}(\xvec)) =
\begin{cases}
1, & S = T_{n,k},\\
0, & S \ne T_{n,k}.
\end{cases}
\]
\end{lemma}
\begin{proof}
   Using the monomial expansion $F_{n,S}(\xvec)=\sum\limits_{S \subseteq D(\alpha) \subseteq [n-1]}M_\alpha(\xvec)$, 
   we have
   \[
   \widehat{\Phi}_k(F_{n,S}(\xvec))=
\sum_{\substack{S \subseteq D(\alpha) \subseteq [n-1] \\
\alpha_i \equiv 0 \pmod{k}}}
(-1)^{n/k - \ell(\alpha)}.
   \]
   If $S$ contains an element not divisible by $k$, the sum is zero.
Otherwise write $\alpha = k\beta$ with $\beta \vDash m$, so
\[
D(\alpha) = k \cdot D(\beta).
\]
This reduces the sum to
\[
\widehat{\Phi}_k(F_{n,S}(\xvec))=\sum_{S' \subseteq D(\beta) \subseteq [m-1]} (-1)^{m - \ell(\beta)},
\]
where $S' = S/k$.
If $U=D(\beta)$, let $U=S'\bigcup W$. Then $\ell(\beta)=|U|+1$ and
\[
\widehat{\Phi}_k(F_{n,S}(\xvec))=(-1)^{m-1-|S'|}\sum_{W\subseteq[m-1]\setminus S'}(-1)^{|W|}=(1-1)^{m-1-|S'|},
\]
This is $1$ when $S'=[m-1]$ and $0$ otherwise. The condition $S'=[m-1]$ is equivalent to $S=T_{n,k}$.
\end{proof}

The theory of $P$-partitions gives
\begin{equation}\label{P-partition expansion}
    X_G=
\sum_{\mathfrak{o} \in AO(G)}
\ \sum_{w \in \mathcal{L}(P_{\mathfrak{o}}, \omega_{\mathfrak{o}})}
F_{n,\mathrm{Des}(w)},
\end{equation}
where $AO(G)$  is the set of acyclic orientations, $P_{\mathfrak{o}} $ 
is the poset induced by the acyclic orientation, 
$\omega_{\mathfrak{o}}$ is a decreasing labeling,
 $\mathcal{L}(P_{\mathfrak{o}}, \omega_{\mathfrak{o}})$ is the set of 
 linear extensions, see~\cite{ShareshianWachs2016}. 
 Applying Lemma~\ref{lemm:Phi_kF} to~\eqref{P-partition expansion} gives the following corollary.

\begin{corollary}
    Let $G = (V,E)$ be a graph on $n$ vertices, and let $k$ be a positive integer. If $k\mid n$, then
\[
\Phi_k(X_G)=
\#\{
(\mathfrak{o}, w) :
\mathfrak{o} \in AO(G), \;
w \in \mathcal{L}(P_{\mathfrak{o}}, \omega_{\mathfrak{o}}), \;
\mathrm{Des}(w) = \{k,2k,\dots,n-k\}
\}.
\]
If $k\nmid n$, then $\Phi_k(X_G)=0$. In particular, $\Phi_1(X_G)=\# AO(G)$. 
\end{corollary}

\subsection{Generalizations}

For each $i \in \mathbb{Z}_{>0}$, define operators $\pi_i$ and $\theta_i$ on $\mathbb{Z}[x_1,x_2,\dots]$ by
\[
\pi_i f = \partial_i(x_i f), \qquad \theta_i f = x_{i+1}\,\partial_i f,
\]
where the divided difference operator $\partial_i$ is given by
\[
\partial_i f = \frac{f - s_i f}{x_i - x_{i+1}},
\]
and $s_i$ acts by swapping $x_i$ and $x_{i+1}$, see~\cite{Macdonald1991}. 

If $w = s_{i_1}\cdots s_{i_k}$ is a reduced expression, define
\[
\pi_w = \pi_{i_1}\cdots \pi_{i_k}, \qquad \theta_w = \theta_{i_1}\cdots \theta_{i_k}.
\]
These definitions are independent of the choice of reduced expression.

\begin{definition}[Key polynomials~\cite{Demazure1974}]
Let $\alpha = (\alpha_1,\alpha_2,\dots)$ be a composition with finitely many nonzero parts. 
Let $\lambda = \mathrm{sort}(\alpha)$ be the partition obtained by rearranging the parts of $\alpha$ in weakly decreasing order, and let $w \in S_\infty$ be the shortest permutation such that
\[
\lambda \cdot w = \alpha.
\]
The \defin{key polynomial} (or \defin{Demazure character}) indexed by $\alpha$ is defined by
\[
\kappa_\alpha = \pi_w(x^\lambda),
\]
where $x^\lambda = x_1^{\lambda_1} x_2^{\lambda_2} \cdots$, and $\pi_w$ is the isobaric divided difference operator associated to $w$.
\end{definition}

\begin{definition}[Atom polynomials~\cite{Lascoux1990Keys}]
Let $\alpha$ be a composition, and let $\lambda = \mathrm{sort}(\alpha)$ and $w \in S_\infty$ be as above. 
The \defin{atom polynomial} indexed by $\alpha$ is defined by
\[
A_\alpha = \theta_w(x^\lambda),
\]
where $\theta_w$ is the isobaric divided difference operator defined using the operators $\theta_i$.
\end{definition}

\begin{example}
The image of a key polynomial need not be key-positive:
\begin{align*}
\rowMap_2(\kappa_{(0,3,1,4)})
=&\kappa_{(0,2,1,1)} + \kappa_{(1,1,0,2)} + \kappa_{(1,1,1,1)} - \kappa_{(1,2,0,1)}\\
  =& A_{(0,2,1,1)} + A_{(1,1,0,2)} + A_{(1,1,1,1)} + A_{(1,1,2)} \\
  +& A_{(1,2,0,1)} + A_{(1,2,1)} + A_{(2,0,1,1)} + A_{(2,1,0,1)} + A_{(2,1,1)}.
\end{align*}
\end{example}

\begin{conjecture}
The row operator applied to a key polynomial is atom-positive.
\end{conjecture}

\textbf{Acknowledgments.} The second author acknowledges financial support from the China Scholarship Council (CSC) and thanks Dun Qiu for helpful discussions. 

\textbf{AI assistance.} The authors used GPT (Codex) to help revise the language.
\bibliographystyle{alpha}
\bibliography{theBib}

@article{AlexanderssonPfannererRubeyUhlin2020,
Author = {Per Alexandersson and Stephan Pfannerer and Martin Rubey and Joakim Uhlin},
Title = {Skew characters and cyclic sieving},
doi = {10.1017/fms.2021.11},
Year = {2021},
volume = {9},
pages = {e41},
publisher={Cambridge University Press},
journal = {Forum of Mathematics, Sigma}
}

@article{Branden2004sg,
  title = {Sign-Graded Posets, Unimodality of  $W$-Polynomials and the {C}harney--{D}avis Conjecture},
  volume = {11},
  ISSN = {1077-8926},
  url = {http://dx.doi.org/10.37236/1866},
  DOI = {10.37236/1866},
  number = {2},
  journal = {The Electronic Journal of Combinatorics},
  publisher = {The Electronic Journal of Combinatorics},
  author = {Br{\"{a}}nd{\'{e}}n,  Petter},
  year = {2004},
  month = nov 
}

@book{Macdonald1995,
	address = {New York},
	author = {Ian G. Macdonald},
	edition = {Second},
	isbn = {0-19-853489-2},
	mrclass = {05E05 (05-02 20C30 20C33 20K01 33C80 33D80)},
	note = {With contributions by A. Zelevinsky, Oxford Science Publications},
	pages = {x+475},
	publisher = {The Clarendon Press, Oxford University Press},
	series = {Oxford Mathematical Monographs},
	title = {Symmetric functions and {H}all polynomials},
	year = {1995}
}

@Book{Macdonald1991,
	title = {Notes on {S}chubert polynomials},
	author = {Ian G. Macdonald},
	year = {1991},
	publisher = {LACIM, Université du Québec à Montréal}
}

@book{Fulton1997,
  title = {Young Tableaux: With Applications to Representation Theory and Geometry},
  doi = {10.1017/cbo9780511626241},
  author = {William Fulton},
  isbn = {978-0511626241},
  series = {London Mathematical Society Student Texts (Book 35)},
  year = {1997},
  publisher = {Cambridge University Press}
}

@book{StanleyEC2,
    author = {Richard P. Stanley},
    day = {04},
    edition = {First},
    howpublished = {Paperback},
    isbn = {0521789877},
    publisher = {Cambridge University Press},
    title = {Enumerative {C}ombinatorics: {V}olume 2},
    url2 = {http://www.worldcat.org/isbn/0521789877},
    doi={10.1017/CBO9780511609589},
    year = {2001}
}

@incollection{KingTolluToumazet2004,
  author = {R. C. King and C. Tollu and F. Toumazet},
  title = {Stretched {L}ittlewood--{R}ichardson coefficients and {K}ostka coefficients},
  journal = {CRM Proceedings and Lecture Notes},
  year = {2004},
  pages = {99--112},
  volume = {34},  
  booktitle = {Symmetry in Physics: In Memory of {R}obert {T}. {S}harp},
  isbn = {978-0-8218-3409-1},
  editor = {P. Winternitz and J. Harnard and C. S. Lam and J. Patera},
  publisher = {AMS / OUP},
  url = {http://eprints.soton.ac.uk/41202/}
}

@article{Mcallister2008,
 author = {Tyrrell B. McAllister},
 title = {Degrees of stretched {K}ostka coefficients},
 journal = {J. Algebraic Comb.},
 volume = {27},
 number = {3},
 month = may,
 year = {2008},
 issn = {0925-9899},
 pages = {263--273},
 numpages = {11},
 doi = {10.1007/s10801-007-0083-2},
 acmid = {1361658},
 publisher = {Kluwer Academic Publishers},
 address = {Hingham, MA, USA}
}

@article{Rassart2004,
 author = {Etienne Rassart},
 title = {A polynomiality property for {L}ittlewood--{R}ichardson coefficients},
 journal = {J. Comb. Theory Ser. A},
 issue_date = {August 2004},
 volume = {107},
 number = {2},
 month = aug,
 year = {2004},
 issn = {0097-3165},
 pages = {161--179},
 numpages = {19},
 doi = {10.1016/j.jcta.2004.04.003},
 acmid = {1032085},
 publisher = {Academic Press, Inc.},
 address = {Orlando, FL, USA},
}

@incollection{Lascoux1990Keys,
    AUTHOR = {Lascoux, Alain and Sch{\"{u}}tzenberger, Marcel-Paul},
     TITLE = {Keys \& standard bases},
 BOOKTITLE = {Invariant theory and tableaux ({M}inneapolis, {MN}, 1988)},
    SERIES = {IMA Vol. Math. Appl.},
    VOLUME = {19},
     PAGES = {125--144},
 PUBLISHER = {Springer, New York},
      YEAR = {1990}
}

@article{ShareshianWachs2016,
  doi = {10.1016/j.aim.2015.12.018},
  url2 = {https://doi.org/10.1016/j.aim.2015.12.018},
  year  = {2016},
  month = jun,
  publisher = {Elsevier {BV}},
  volume = {295},
  number = {4},
  pages = {497--551},
  author = {John Shareshian and Michelle L. Wachs},
  title = {Chromatic quasisymmetric functions},
  journal = {Advances in Mathematics}
}

@article{Stanley1995,
author = {Richard P. Stanley},
title = {A Symmetric Function Generalization of the Chromatic Polynomial of a Graph},
journal = {Advances in Mathematics},
volume = {111},
number = {1},
pages = {166--194},
year = {1995},
issn = {0001-8708},
doi = {10.1006/aima.1995.1020},
url2 = {http://www.sciencedirect.com/science/article/pii/S0001870885710201}
}

@article{GuayPaquet2013,
  author = {Mathieu Guay-Paquet},
  title = {A modular law for the chromatic symmetric functions of $(3+1)$-free posets},
  eprint = {1306.2400},
  journal = {arXiv e-prints},
  year = {2013}
}

@article{Demazure1974,
  doi = {10.24033/asens.1261},
  url2 = {https://doi.org/10.24033/asens.1261},
  year  = {1974},
  publisher = {Societe Mathematique de France},
  volume = {7},
  number = {1},
  pages = {53--88},
  author = {Michel Demazure},
  title = {D{\'{e}}singularisation des vari{\'{e}}t{\'{e}}s de Schubert g{\'{e}}n{\'{e}}ralis{\'{e}}es},
  journal = {Annales scientifiques de l'{\'{E}}cole normale sup{\'{e}}rieure}
}

@article{Gasharov1996,
title = {Incomparability graphs of $(3+1)$-free posets are $s$-positive},
journal = {Discrete Mathematics},
volume = {157},
number = {1},
pages = {193--197},
year = {1996},
issn = {0012-365X},
doi = {10.1016/S0012-365X(96)83014-7},
url2 = {http://www.sciencedirect.com/science/article/pii/S0012365X96830147},
author = {Vesselin Gasharov}
}

@article{AthanasiadisWagner2024,
  author = {Christos A. Athanasiadis and David G. Wagner},
  title = {Veronese sections and interlacing matrices of polynomials and formal power series},
  journal = {arXiv preprint arXiv:2404.12989},
  year = {2024},
  doi = {10.48550/arXiv.2404.12989},
  url = {https://arxiv.org/abs/2404.12989}
}

@article{EgeciogluRemmel1990,
  doi = {10.1080/03081089008817966},
  url2 = {https://doi.org/10.1080/03081089008817966},
  year = {1990},
  month = jan,
  publisher = {Informa {UK} Limited},
  volume = {26},
  number = {1-2},
  pages = {59--84},
  author = {{\"{O}}mer E{\u{g}}ecio{\u{g}}lu and Jeffrey B. Remmel},
  title = {A combinatorial interpretation of the inverse {K}ostka matrix},
  journal = {Linear and Multilinear Algebra}
}

@article{McNamara2006,
  doi = {10.1016/j.aim.2005.07.011},
  url2 = {https://doi.org/10.1016/j.aim.2005.07.011},
  year  = {2006},
  month = sep,
  publisher = {Elsevier {BV}},
  volume = {205},
  number = {1},
  pages = {275--312},
  author = {Peter McNamara},
  title = {Cylindric skew {S}chur functions},
  journal = {Advances in Mathematics}
}

@article{KaliszewskiMorse2019,
  doi = {10.1016/j.ejc.2019.05.006},
  url2 = {https://doi.org/10.1016/j.ejc.2019.05.006},
  year = {2019},
  month = oct,
  publisher = {Elsevier {BV}},
  volume = {81},
  pages = {354--377},
  author = {Ryan Kaliszewski and Jennifer Morse},
  title = {Colorful combinatorics and {M}acdonald polynomials},
  journal = {European Journal of Combinatorics}
}

@article{GrojnowskiHaiman2006,
author = {Ian Grojnowski and Mark Haiman},
year = {2006},
title = {Affine {H}ecke algebras and positivity of {LLT} and {M}acdonald polynomials},
journal = {Preprint},
url = {https://math.berkeley.edu/~mhaiman/ftp/llt-positivity/new-version.pdf}
}

@incollection{Gessel1984,
    AUTHOR = {Ira M. Gessel},
     TITLE = {Multipartite ${P}$-partitions and inner products of skew {S}chur functions},
 BOOKTITLE = {Combinatorics and algebra ({B}oulder, {C}olo., 1983)},
    SERIES = {Contemp. Math.},
    VOLUME = {34},
     PAGES = {289--317},
 PUBLISHER = {Amer. Math. Soc., Providence, RI},
      YEAR = {1984},
   MRCLASS = {05A17 (20C30)},
  MRNUMBER = {777705},
MRREVIEWER = {Jacques D{\'{e}}sarm{\'{e}}nien},
       DOI = {10.1090/conm/034/777705},
       URL2 = {https://doi.org/10.1090/conm/034/777705},
}

@book{Sagan2001,
  doi = {10.1007/978-1-4757-6804-6},
  url2 = {https://doi.org/10.1007/978-1-4757-6804-6},
  year = {2001},
  publisher = {Springer New York},
  author = {Bruce E. Sagan},
  title = {The Symmetric Group}
}

@article{Albion2022,
Author = {Seamus P. Albion},
Title = {Universal characters twisted by roots of unity},
Year = {2022},
journal = {Algebr. Comb. 6 (2023), no. 6, 1653-1676},
doi = {10.5802/alco.320},
  url = {https://doi.org/10.5802/alco.320}
}

@article{LLT1997,
  author = {A. Lascoux and B. Leclerc and J.-Y. Thibon},
  title = {Ribbon tableaux, {H}all--{L}ittlewood functions, quantum affine algebras, and unipotent varieties},
  journal = {J. Math. Phys.},
  volume = {38},
  year = {1997},
  pages = {1041--1068}
}

@article{Hikita2026,
  author  = {Tatsuyuki Hikita},
  title   = {A proof of the {Stanley--Stembridge} conjecture},
  journal = {J. Amer. Math. Soc.},
  year    = {2026},
  note    = {DOI: https://doi.org/10.1090/jams/1090},
}

@article{AMDEBERHAN_ONO_SINGH_2025,
  title = {Derivatives of theta functions as traces of partition {E}isenstein series},
  volume = {258},
  DOI = {10.1017/nmj.2024.30},
  journal = {Nagoya Mathematical Journal},
  author = {Amdeberhan, Tewodros and Ono, Ken and Singh, Ajit},
  year = {2025},
  pages = {284--295}
}

@article{StanleyFlagSymmetric96,
  title = {Flag-symmetric and locally rank-symmetric partially ordered sets},
  author = {Richard P. Stanley},
  year = {1996},
  url = {http://www.combinatorics.org/Volume_3/Abstracts/v3i2r6.html},
  researchr = {https://researchr.org/publication/Stanley96},
  cites = {0},
  citedby = {0},
  journal = {Electr. J. Comb.},
  volume = {3},
  number = {2},
}

@misc{Tao2017,
  author = {Terence Tao},
  title = {Continuous analogues of the Schur and skew Schur polynomials},
  year = {2017},
  howpublished = {\url{https://terrytao.wordpress.com/2017/09/05/continuous-analogues-of-the-schur-and-skew-schur-polynomials/}}
}

@article{Prasad2018,
  author = {Amritanshu Prasad},
  title = {A Timed Version of the Plactic Monoid},
  journal = {arXiv preprint arXiv:1806.04393},
  year = {2018}
}

@incollection{OConnell2014,
  author = {Neil O'Connell},
  title = {Whittaker Functions and Related Stochastic Processes},
  booktitle = {Random Matrix Theory, Interacting Particle Systems, and Integrable Systems},
  series = {MSRI Publications},
  volume = {65},
  pages = {385--410},
  publisher = {Cambridge University Press},
  year = {2014}
}

@article{AmdeberhanShareshianStanley2026,
  year = {2026},
  author = {Tewodros Amdeberhan and John Shareshian and Richard P. Stanley},
  title = {Sprout Symmetric Functions, Part 1},
  journal = {arXiv preprint arXiv:2605.27828},
  eprint = {2605.27828},
  archivePrefix = {arXiv},
  note = {Version 1}
}

@article{AlexanderssonDai2026Epositive,
  author = {Per Alexandersson and Lilan Dai},
  title = {The {$e$}-positivity of {$\mathrm{rowDiv}_k(e_\mu)$}},
  journal = {arXiv preprint arXiv:2610.26303},
  year = {2026},
  eprint = {2610.26303},
  archivePrefix = {arXiv},
  primaryClass = {math.CO},
  doi = {10.48550/arXiv.2610.26303},
  url = {https://arxiv.org/abs/2610.26303v1},
  note = {Version 1}
}

\end{document}